\documentclass[a4paper,reqno]{amsart}
\usepackage[english]{babel}
\usepackage{amsmath,amsthm,amssymb,amsfonts}

\usepackage{amsaddr}
\usepackage{enumitem}
\usepackage{comment}
\usepackage{hyperref}
\usepackage{mathtools}
\usepackage[T1]{fontenc}
\usepackage[utf8]{inputenc}
\hypersetup{
 colorlinks,
 linkcolor={blue!90!black},
 citecolor={red!80!black},
 urlcolor={blue!50!black}
}
\usepackage{tikz}
\usepackage{tikz-cd}
\usepackage{graphicx}
\usepackage[all,cmtip]{xy}
\usepackage{mleftright}
\usepackage{booktabs}
\usepackage{cite}

\theoremstyle{theorem}
\newtheorem{theorem}{\sc \textbf{Theorem}}[section]  
\newtheorem{proposition}[theorem]{\sc \textbf{Proposition}}   
\newtheorem{corollary}[theorem]{\sc \textbf{Corollary}}        
\newtheorem{lemma}[theorem]{\sc \textbf{Lemma}}

\renewcommand{\approx}{ \asymp}

\theoremstyle{remark}

\newtheorem{remark}[theorem]{\sc \textbf{Remark}}

\newcommand{\Ls}{\mathcal{L}}
\newcommand{\Rs}{\mathcal{R}}

\def\bX{{\mathbb X}}

\numberwithin{equation}{section}

\newcommand{\N}{\mathbb{N}}

\newcommand{\R}{\mathbb{R}}

\def\e{\mathrm{e}}

\newcommand{\dd}{d}

\usepackage{pdfpages}
\newcommand{\x}{\underline{x}}

\newcommand{\Dom}{\mathrm{Dom}}

\makeatletter
\@namedef{subjclassname@1991}{Mathematics Subject Classification}
\usepackage{enumerate}
\graphicspath{{images/}}


\author[Tommaso Bruno]{Tommaso Bruno}
\address{Dipartimento di Matematica, Universit\`a degli Studi di Genova\\ Via Dodecaneso 35, 16146 Genova, Italy}
\email{brunot@dima.unige.it}

\author[Federico Santagati]{Federico Santagati}
\address{Dipartimento di Scienze Matematiche ``G.\ L.\ Lagrange'', Politecnico di Torino\\ Corso Duca degli Abruzzi 24, 10129 Torino, Italy}
\email{federico.santagati@polito.it}

\begin{document}
\title[Heat kernel on Damek--Ricci Spaces]{Optimal Heat kernel bounds and asymptotics \\ on Damek--Ricci Spaces}
	\maketitle

\begin{abstract}
We give optimal bounds for the radial, space and time derivatives of arbitrary order of the heat kernel of the Laplace--Beltrami operator on Damek--Ricci spaces. In the case of symmetric spaces of rank one, these complete and actually improve conjectured estimates by Anker and Ji. We also provide asymptotics at infinity of all the radial and time derivates of the kernel. Along the way, we provide sharp bounds for all the derivatives of the Riemannian distance and obtain analogous bounds for those of the heat kernel of the distinguished Laplacian. 
\end{abstract}

\section{Introduction}
In this paper we obtain optimal bounds for all the derivatives of the heat kernel $h_{t}$ of the Laplace--Beltrami operator $\Ls$ on Damek--Ricci spaces. These are one-dimensional extensions $S= N \ltimes \R^{+}$ of an H-type group $N$ and include all symmetric spaces of noncompact type of rank one and hence all real hyperbolic spaces. Heat kernel estimates in such framework, for $h_{t}$ itself, for some of its derivatives and in different degrees of generality, have been object of longstanding investigations; see the milestones~\cite{DM,ADY,AJ} but also~\cite{GM, LS, FP, SV}, and the references therein. Our estimates improve and extend all those known so far for the derivatives of $h_{t}$ on Damek--Ricci spaces, and in the case of symmetric spaces of noncompact type of rank one they complete (actually  improve) conjectured bounds by Anker and Ji~\cite{AJ}. We also transfer some of our estimates of $h_{t}$ to analogous estimates of the heat kernel of the so-called distinguished Laplacian $\Delta$ on $S$, by means of a well-known equivalence of $\Delta$ and $\Ls$ at the $L^{2}$ level. 

The derivatives of $h_{t}$ which we consider are of three types:  \emph{radial} derivatives,  \emph{time} derivatives, and \emph{space} derivatives, namely those along the vector fields in the Lie algebra of $S$. Of the first two, we also provide the asymptotic behavior at infinity. On the way to estimate the third ones, we establish sharp bounds for all the derivatives of the left-invariant Riemannian distance on $S$, which have independent interest.

It is well known that heat kernel estimates offer a great number of applications~\cite{SC}. We conclude our paper by showing the weak-type $(1,1)$ of certain maximal operators associated with $h_{t}$ and its derivatives, and by obtaining spectral properties of Ornstein--Uhlenbeck operators on $S$. We envisage more applications in the next future, ranging from estimates for the derivatives of the Bessel–Green–Riesz and Poisson kernels of the semigroups generated by $\Ls$~\cite{AJ} to the study of function spaces~\cite{BCF, B} and related geometric inequalities~\cite{BPV2}, to boundedness results for singular integrals associated with $\Ls$~\cite{GS,MV,A1}, just to name a few.

We now describe in detail the setting and the main results of the paper.

\subsection{Damek--Ricci spaces. Preliminaries}

An H-type group $N$ is a 2-step stratified group whose Lie algebra $\mathfrak{n}$ is endowed with an inner product $(\, \cdot\,,\, \cdot \,)$ such that
\begin{itemize}
	\item if $\mathfrak{z}$ is the centre of $\mathfrak{n}$ and $\mathfrak{v}=\mathfrak{z}^\perp$, then $[\mathfrak{v},\mathfrak{v}]=\mathfrak{z}$;
	\item for every $Z \in \mathfrak{z}$, the map $J_Z\colon \mathfrak{v} \to \mathfrak{v}$, 
	\[(J_Z X,Y)=( Z, [X,Y]) \qquad \forall X,Y\in \mathfrak{v},\]
	is an isometry whenever $(Z, Z)=1$.
\end{itemize}
In particular, $\mathfrak{n}$ stratifies as $\mathfrak{v} \oplus \mathfrak{z}$. We shall realize an H-type group $N$ as $\R^{\mu}\times \R^\nu $, for some $\mu,\nu\in \N$, via the exponential map. More precisely, we shall denote by $(x,z)$ the elements of $N$, where  $x\in \R^{\mu}$ and $z \in \R^\nu$. We denote by  $(e_1,\dots,e_{\mu})$ and $(u_1,\dots,u_\nu)$ the standard bases of $\R^{\mu}$ and $\R^{\nu}$ respectively. Under this identification, the Haar measure on $N$ is the Lebesgue measure $\dd x \, \dd z$. The maps $\{J_Z \colon Z\in \mathfrak{z}\}$ are identified with $\mu \times \mu$ skew-symmetric matrices $\{J_z \colon z\in \R^\nu\}$ which are orthogonal whenever $|z|=1$. This identification endows $\R^{\mu}\times \R^{\nu}$ with the group law
\[(x,z)\cdot (x',z') = \bigg(x+x',z+z' + \frac{1}{2} \sum_{k=1}^\nu (J_{u_k}x,x') u_k\bigg).\]
We recall that $\mu$ is always even, and we denote by $Q={(\mu+2\nu)}/{2}$ (half) the homogeneous dimension of $N$.

\smallskip

A Damek--Ricci space~\cite{DR1,DR2} is then the one-dimensional extension $S= N \ltimes \R^{+}$ of an H-type group $N$ obtained by making $\R^+$ act on $N$ by homogeneous dilations. The aforementioned notation for $N$ will be fixed throughout, and the elements of $S$ will be correspondingly denoted by $(x,z,a)$, where $a\in \R^{+}$.  The resulting group law on $S$ is then given by
\[
(x,z,a) \cdot (x',z',a') = \bigg(x + \sqrt{a}x', z+az'+ \frac{1}{2}\sqrt{a} \sum_{k=1}^\nu (J_{u_k}x,x') u_k, aa'\bigg),
\]
and its identity is $e=(0,0,1)$. For notational convenience, the generic element of $S$ will sometimes be denoted also by $\x = (x,z,a)$. The Lie algebra of $S$ can be identified with $\mathfrak{s}= \mathfrak{v} \oplus \mathfrak{z} \oplus \R$ and its Lie bracket satisfies
\begin{equation}\label{commcampi}
[(X,Z,b),(X',Z',b')] = \big( \tfrac{1}{2}bX' - \tfrac{1}{2}b'X, bZ'-b'Z+[X,X'],0\big).
\end{equation}
We endow $S$  with the left-invariant Riemannian metric induced by the product
\[
\langle (X,Z,b), (X',Z',b')\rangle = (X,X') +  (Z,Z' ) + bb'
\]
on $\mathfrak{s}$. The associated Riemannian left-invariant (Haar) measure on $S$ is 
\[
\dd \lambda (x,z,a) = a^{-Q-1}\dd x\, \dd z\, \dd a.
\]
The modular function is $\delta (x,z,a) = a^{-Q}$, and the associated right Haar measure is $\dd \rho (x,z,a) = a^{-1}\dd x\, \dd z\, \dd a $. We shall denote by $n=\mu+\nu+1$ the dimension of $S$. We recall that the Riemannian manifold $S$ is locally doubling but not globally doubling, as the measure of its balls grows exponentially for large radii; cf.~\cite[(1.18)]{ADY}.

\smallskip

Denote by $d \colon S\times S \to [0,\infty)$ the Riemannian distance. A function $f$ on $S$ is said to be radial if it depends only on the distance from the identity, i.e., if there exists a function $f_0$ on $[0,+\infty)$ such that $f(x,z,a)=f_0(r)$, where 
\[
r=r(x,z,a) = d ((x,z,a),e).
\]
We abuse the notation and write $f(r)$ instead of $f_0(r)$; and the terminology by referring to $r\colon S \to [0,\infty)$ as the \emph{distance function}.  A more detailed account on $r$ and further details will be given in due course, in particular in Section~\ref{sec:distance} below. 

\smallskip

Let $\bX_0,\bX_1,\dots, \bX_{\mu}$ and $ \bX_{\mu+1}, \dots, \bX_{n-1}$ be the left-invariant vector fields  on $S$ which agree with $\partial_{a}, \partial_{x_1}, \dots, \partial_{x_\mu}$ and $\partial_{z_{1}},\dots, \partial_{z_{\nu}}$  respectively at the identity. In other words, given $f$ in $ C^{\infty}_c(S)$,
\begin{align*}
\bX_0f(x,z,a)
&= \frac{\dd}{\dd t }\Big\lvert_{t =0}f\big((x,z,a) \cdot (0,0,e^t)\big),
\end{align*}
while, for $\ell=1,\dots,\mu$,
\begin{align*}
\bX_\ell f(x,z,a)
& = \frac{\dd}{\dd t }\Big\lvert_{t =0}f\big((x, z,a) \cdot (te_\ell ,0,1)\big),
\end{align*}
and, for $k=1,\dots,\nu$,
\begin{align*}
\bX_{\mu +k} f(x,z,a)
&= \frac{\dd}{\dd t }\Big\lvert_{t =0}f\big((x, z,a) \cdot (0,tu_k,1)\big).
\end{align*}
Simple computations then lead to 
\begin{align*}
\bX_0 = a\, \partial_a, \qquad \bX_\ell = \sqrt{a}\, \bigg( \partial_{x_\ell} + \frac{1}{2}\sum_{k=1}^{\nu}( J_{u_k} x,e_\ell)\partial_{z_k}\bigg), \qquad \mathbb{X}_{\mu+k} = a \,\partial_{z_k},
\end{align*}
for $\ell=1,\dots,\mu$ and $k=1,\dots,\nu$. In particular, the vector fields $\bX_j$ do not involve derivatives in the variable $a$ when $j\neq 0$. The non-negative Laplace--Beltrami operator 
on $S$ takes the form, see~\cite{Damek},
\[
\Ls =-\sum_{j=0}^{n-1} \bX_j^2 +Q \bX_0.
\]
The operator $\Ls$ is essentially self-adjoint on $L^{2}(\lambda)$ and its $L^{2}$-spectrum lies in $[Q^2/4, \infty)$. We shall denote by $\e^{-t\Ls}$ the heat semigroup generated by $\Ls$ and by $h_t$ its radial right-convolution kernel on $S$, i.e.\ its \emph{heat kernel}. In particular
 \[
\e^{-t\Ls}f= f\ast h_t
\]
where $\ast $ denotes the convolution on $S$. Following the aforementioned abuse, we shall sometimes interpret $h_{t}$ as a function on $S$, and other times as a function on $[0,\infty)$, namely $h_{t}(\x) = h_{t}(r(\x))$.

\begin{remark}
If one considers the degenerate case $N=\R^{\nu}$ among H-type groups, namely $\mathfrak{v}=\{0\}$ (whence $\mu=0$) and $\mathfrak{z}= \R^{\nu}$, then the so-called ``$ax+b$'' groups and hence all real hyperbolic spaces fall within the above framework. In this order of ideas, all symmetric spaces of noncompact type and rank one are Damek--Ricci spaces. All of our results do include such degenerate case, though the proofs might require minor modifications (actually simplifications, as $\mu=0$ and the $x$-variable should always be neglected). No further comment on this will be given and the few details will be left to the interested reader.
\end{remark}

\subsection{More on the heat kernel. Main results}
The spherical analysis available on $S$ provides a rather explicit expression of $h_{t}$. Denote, for future convenience, 
\begin{equation}\label{ch}
\mathfrak{c}_{k}= 2^{-\mu-\frac{\nu}{2}-1-k} \pi^{-\frac{n}{2}}, \qquad k=0,1,\dots
\end{equation}
If $\nu$ is even, then
\begin{align}\label{heatkerneleven}
    h_t(r) = \mathfrak{c}_{0} t^{-\frac12} \e^{-\frac{Q^2}{4}t} \left( -\frac{1}{\sinh r} \frac{\partial}{\partial r}\right)^{\frac{\nu}{2}} \left( -\frac{1}{\sinh \frac{r}{2}} \frac{\partial}{\partial r}\right)^{\frac{\mu}{2}} \e^{-\frac{r^2}{4t}}
\end{align}
while if $\nu$ is odd
\begin{align}\label{heatkernelodd}
\nonumber h_t(r) &= \mathfrak{c}_{0} \pi^{-\frac{1}{2}}t^{-\frac12} \e^{-\frac{Q^2}{4}t} \\
 &\times \int_r^\infty
 \frac{\sinh s}{\sqrt{\cosh s - \cosh r}} \left( -\frac{1}{\sinh s} \frac{\partial}{\partial s}\right)^{\frac{\nu+1}{2}} \left( -\frac{1}{\sinh \frac{s}{2}} \frac{\partial}{\partial s}\right)^{\frac{\mu}{2}} \e^{-\frac{s^2}{4t}} \, \dd s.
\end{align}
This difference is due to the underlying presence of an inverse Abel transform; we will not go into details here, but rather pick~\eqref{heatkerneleven} and~\eqref{heatkernelodd} as starting points of our analysis. The reader can find a comprehensive discussion in~\cite{ADY}.

Precise bounds for $h_{t}$ and its gradient were obtained in~\cite[\S 5]{ADY}. In particular
\begin{align}\label{boundsht}
h_t(r) \approx t^{-\frac32} (1+r)  \left(1+ \frac{1+r}{t}\right)^{\frac{n-3}{2}} \e^{-\frac{Q^2}{4}t - \frac{Q}{2}r - \frac{r^2}{4t}}
\end{align}
and
\begin{align}\label{boundsgradht}
 |\nabla \, h_t(r)|  &= \bigg| \frac{\partial}{\partial r} h_t(r)\bigg| \approx t^{-\frac32}  r  \left(1+ \frac{1+r}{t}\right)^{\frac{n-1}{2}} \e^{-\frac{Q^2}{4}t - \frac{Q}{2}r - \frac{r^2}{4t}}.
\end{align}
Here and all throughout, $ v \approx w$ for two positive functions $v$ and $w$ means that there exists a constant $C$ (depending only on structural constants of $S$, and possibly on other circumstantial parameters) such that $C^{-1}v \leq w \leq Cv$. Analogously, we shall write $v\lesssim w$ if there exists such a $C$ with the property that $v\leq Cw$. 

Observe in particular that the upper bound in~\eqref{boundsgradht} is equivalent, by~\eqref{boundsht}, to
\begin{equation}\label{boundder}
 \bigg| \frac{\partial}{\partial r} h_t(r)\bigg| \lesssim h_{t}(r) \times
 \begin{cases} 
 1+ \frac{1}{\sqrt{t}} + \frac{r}{t}   &\text{if $r >1$} \\ 
 r \Big(1+\frac{1}{t}\Big) \qquad &\text{if $r\leq 1$}.
 \end{cases}
\end{equation}

One of the aims of this paper is to generalize~\eqref{boundder} by considering derivatives of arbitrary order. In particular, for $k\in \N$ and a multi-index $J= \{ 0, \dots, n-1\}^{k}$, we shall write $J_{j}$ for $J$'s $j$-th component and $|J|$ for its length, i.e., $|J|=k$. By $\bX_{J}$ we shall mean the left-invariant differential operator
\[
\bX_{J}= \bX_{J_{1}}\bX_{J_{2}} \cdots \bX_{J_{k}}.
\]

 Let us define, for $k\in \N$ and $t,r>0$, the functions
\begin{align*}
\Psi_{k}(r,t)&= 
 \begin{cases} 
 \Big(1+ \frac{1}{\sqrt{t}} + \frac{r}{t}\Big)^k  &\text{if $r >1$, or $r\leq 1$ and $k$ is even} \\ 
 r \Big(1+ \frac{1}{\sqrt{t}} + \frac{r}{t}\Big)^{k-1}\Big(1+\frac{1}{t}\Big) \qquad &\text{if $r\leq 1$ and $k$ is odd,}
 \end{cases}
\end{align*}
and
\begin{align*}
\tilde \Psi_{k}(r,t)&= 
 \begin{cases} 
 \Big(1+ \frac{1}{\sqrt{t}} + \frac{r}{t}\Big)^k  &\text{if $r >1$, or $r\leq 1$ and $k$ is even} \\ 
\Big(1+ \frac{1}{\sqrt{t}} + \frac{r}{t}\Big)^{k-1}\Big(1+\frac{r}{t}\Big) \qquad &\text{if $r\leq 1$ and $k$ is odd.}
 \end{cases}
\end{align*}
Observe that $\tilde \Psi_{k}(r,t) = \Psi_{k}(r,t)$ if $r> 1$ or if $r\leq 1$ and $k$ is even, while $\Psi_{k}(r,t) \leq \tilde \Psi_{k}(r,t) $ in the remaining case. We shall also need, for $p,q\in \N$, the function
\begin{equation}\label{thetapq}
\Theta_{p,q}(r) = \bigg(\frac{r}{1+r} \frac{\e^{r/2}}{2\sinh(r/2)}\bigg)^{p} \bigg(\frac{r}{1+2r} \frac{\e^{r}}{\sinh(r)}\bigg)^{q}, \qquad r>0.
\end{equation}
Notice that $\lim_{r\to 0^{+}} \Theta_{p,q}(r) = \lim_{r\to \infty} \Theta_{p,q}(r)=1$. Our main results are as follows.

\begin{theorem}\label{teo:main}
Suppose $m,k\in \N$. Then the following holds.
\begin{itemize}
\item[(1)] There exists $C_{k}>0$ such that
\[
\bigg| \frac{\partial^k}{\partial r^k} h_t(r)\bigg|  \leq C_{k} \Psi_{k}(r,t) \, h_{t}(r), \qquad \forall \, t,r>0.
\]
\item[(2)]  There exists $C_{m,k}>0$ such that for all $J\in \{0,\dots, n-1\}^{k}$
\[
\bigg|  \frac{\partial^{m} }{\partial t^{m}}  \bX_{J} h_{t}(r)\bigg| \leq C_{m,k} \tilde \Psi_{2m+k}(r,t) \, h_{t}(r), \qquad \forall \, t,r>0.
\]
If in particular  $J = \{j\}^{k}$ for some $j\in\{0,\dots, n-1\}$, then
\[
|\bX_j^{k} h_{t}(r)| \leq C_{0,k}  \Psi_{k}(r,t) \, h_{t}(r), \qquad \forall \, t,r>0.
\]
\item[(3)] If $(1+r)/t \to \infty$, then
    \begin{align*}
    &\frac{\partial^k}{\partial r^k} h_t(r) = (-1)^{k} \mathfrak{c}_{0} t^{-\frac12} \e^{-\frac{Q^2}{4}t -\frac{r^2}{4t}-\frac{Q}{2} r }(\e^{-r}\sinh r)^{k}\\
   &\qquad \qquad\qquad \times   \bigg(\frac{1+r}{t}\bigg)^{\frac{\mu}{2}}\bigg(\frac{\frac{1}{2} +r}{t}\bigg)^{k + \frac{\nu}{2}} \Theta_{\frac{\mu}{2},k+\frac{\nu}{2}}(r) \bigg[ 1+ O\bigg( \sqrt{\frac{t}{1+r}}\bigg)\bigg],
\end{align*} 
while if $r\to \infty$ and $t>0$ is fixed, then
\begin{align*}  
   \frac{\partial^{m} }{\partial t^{m}}   \frac{\partial^{k} }{\partial r^{k}} h_{t} (r)
    &\sim (-1)^{k}  \mathfrak{c}_{k+2m} \, t^{-\frac{1}{2}}\, \e^{-\frac{Q^2}{4} t}  \, \e^{-\frac{r^{2}}{4t}}\, \e^{-\frac{Q}{2}r }  \bigg(\frac{r}{t} \bigg)^{ \frac{\mu}{2} + \frac{\nu}{2}+k+2m}.
    \end{align*}
\end{itemize}
\end{theorem}

Along the way to Theorem~\ref{teo:main}~(2), we obtain sharp bounds for all the derivatives of the Riemannian distance.
\begin{theorem}\label{teo:derd}
Suppose $k\in \N$. There exists a constant $C_{k}>0$ such that, for all $J \in  \{ 0,\dots, n-1\}^k$,
\[
|\bX_J r|  \leq C_{k} \bigg( \frac{1+r}{r} \bigg)^{k-1}.
\]
\end{theorem}

The results above extend all the estimates of this kind previously known on Damek--Ricci spaces or on real hyperbolic spaces. A few comments may help to clarify our results. 

\begin{itemize}
\item[(i)] The estimates for the derivatives of order $\geq 2$ in Theorems~\ref{teo:main}~(1) and~\ref{teo:derd} seem to be the very first of their kind on any Damek--Ricci space, actually on any Lie group of exponential growth. Those for the space derivatives of $h_{t}$ contained in Theorem~\ref{teo:main}~(2) are obtained by a (nontrivial) combination of these, by means of a new algorithm which might have independent interest; see in particular Proposition~\ref{propodd} below. 
\item[(ii)] The time derivatives contained in Theorem~\ref{teo:main}~(2) generalize those obtained in~\cite[(5.18)]{LS} on real hyperbolic spaces to all Damek--Ricci spaces; they also imply those of~\cite[Theorem 1.1]{FP}. In the case of the first-order time derivative, a slightly better estimate is indeed available; on real hyperbolic spaces, this was already known~\cite[(5.16)]{LS}, and we extend it to Damek--Ricci spaces (Proposition~\ref{prop:LS} below).
\item[(iii)] The space derivatives contained in Theorem~\ref{teo:main}~(2) prove~\cite[Conjecture 3.6]{AJ} in the case of symmetric spaces of noncompact type of rank one (in higher rank the conjecture is still open). For derivatives of \emph{odd} order, Theorem~\ref{teo:main}~(2) is actually strictly better than the aforementioned conjecture, as the case $r=t \to 0^{+}$ shows. This should not surprise: in the Euclidean setting itself, odd-order derivatives show a ``better'' behavior at the origin, as they vanish at $0$, while even-order derivatives do not. It is rather more surprising, then, that $\tilde \Psi$ appears in place of $\Psi$ unless the derivatives are all along the same direction. In case of mixed derivatives, however, this cannot be improved and the distinction is really needed. Further comments will be given in due course, see in particular Remark~\ref{rem:nor} below. 
\item[(iv)] The bounds of Theorem~\ref{teo:derd} should be compared with those of the derivatives of the Euclidean distance on $\R^{n}$. While their singular behavior for $r\to 0^{+}$ is the same, interestingly, no decay appears for $r \to \infty$ on Damek--Ricci spaces. This is not a defect in our argument: we show in Proposition~\ref{sharpinf} that our estimate is sharp.
\item[(v)] To the best of our knowledge, the asymptotic expansions for the derivatives of $h_{t}$ contained in Theorem~\ref{teo:main}~(3) appear here for the first time. To prove the first part of the statement, we borrow some ideas from~\cite{GM0, GM}, though in the case $r\to \infty$ we can provide also a different argument (see Remark~\ref{GMsimplified}). A quite deeper analysis seems to be needed to obtain the estimates for $(1+r)/t \to \infty$ for the time derivatives, at least when the order is larger than $1$. Since this would go out of the scopes of the present paper, we do not tackle this problem here.
\end{itemize}

We conclude this section by describing the structure of the paper, which is as follows.  Section~\ref{sec: 2} is devoted to the study of the radial derivatives of the heat kernel. These will not only prove Theorem~\ref{teo:main}~(1), but also lay the ground for the remainder of the paper. Section~\ref{sec:distance} contains the estimates for the derivatives of the Riemannian distance which will prove Theorem~\ref{teo:derd}. In Section~\ref{sec: 4}, we combine Theorem~\ref{teo:main}~(1) and Theorem~\ref{teo:derd} to get the bounds for the space and time derivatives of $h_{t}$ stated in Theorem~\ref{teo:main}~(2). We also obtain analogous estimates for the heat kernel of the distinguished Laplacian. In Section~\ref{sec: 5} we complete the proof of Theorem~\ref{teo:main} by obtaining the asymptotic expansions in~(3), and in the last Section~\ref{sec: 7} we provide two applications of our results.

\section{Radial derivatives}\label{sec: 2}
    
The aim of this section is to prove Theorem~\ref{teo:main}~(1) and its optimality. All throughout the paper, we shall denote by $\Rs$ the differential operator
\[
\Rs= \frac{1}{\sinh r }\frac{\partial}{\partial r}
\]
on $C^{\infty}((0, +\infty))$, and for $a,b\in \N$, by $p_{a,b}$ the function on $(0,\infty)$ given by $p_{a,b}(r) = \sinh^{a}r \cosh^{b}r$. Observe that for $a,b\geq 1$
\[
\frac{\partial}{\partial r} p_{a,b} = a p_{a-1,b+1} + b p_{a+1,b-1}.
\]
We begin with a lemma. As a general rule in the paper, an empty sum (i.e., a sum for which there is no index satisfying the given conditions) is supposed to be zero. For any positive $\alpha$, we denote by $[\alpha]$ its integer part. 

\begin{lemma}\label{lem: radial derivates} Suppose $k\in \N$. There exist non-negative real numbers  $b_{\ell,j}=b_{\ell,j}(k)$ and $d_{\ell,j}=d_{\ell,j}(k)$,  for $j=1,\dots, k$ and $\ell= 1, \dots, [k/2]+1$, such that:
\begin{itemize}
    \item[i)] if $k $ is odd, then
\begin{align*}
  \frac{\partial^k}{\partial r^k}&=\sum_{j=1}^{(k+1)/2} \bigg(\sum_{1\leq 2\ell+1 \le j} d_{\ell,j} \, p_{2\ell+1, j- 2\ell-1} \bigg)\Rs^{j}\\ & \qquad \qquad \qquad + \sum_{j =(k+3)/2}^{k} \!\!\! \bigg(\sum_{0\le 2\ell \le k-j} b_{\ell,j} \, p_{2j-k+2\ell, k-j-2\ell } \bigg) \Rs^{j};
\end{align*}
    \item[ii)] if $k$ is even, then
\begin{align*}
     \frac{\partial^k}{\partial r^k}&= \sum_{j=1}^{k/2} \bigg( \sum_{0\le 2\ell\le j} d_{\ell, j} \, p_{2\ell, j-2\ell} \bigg)\Rs^{j} + \sum_{j=\frac{k}{2}+1}^{k} \!\! \bigg(\sum_{0\le 2\ell \le k-j} b_{\ell,j} \, p_{2j-k+2\ell, k-j-2\ell} \bigg) \Rs^{j}.
\end{align*}
   
\end{itemize}
\end{lemma} 
\begin{proof}[Proof of Lemma \ref{lem: radial derivates}]
With the straightforward observation that $ \frac{\partial}{\partial r} \Rs^k= \sinh r \Rs^{k+1}$ for all integers $k\geq 0$, the cases $k=1,2$ can be easily checked. Thus $i)$ and $ii)$ hold in these cases. The remaining part is a tedious, but elementary, proof by induction, and we omit the details.
\end{proof} 

The key ingredient to prove Theorem~\ref{teo:main}~(1) is the following lemma.  We also note that combined with~\eqref{boundsht} it provides a direct proof of~\eqref{boundsgradht}.
\begin{lemma}\label{lem: heatderapprox}
Suppose $k \in \mathbb N$. Then
\[
    (-1)^k \Rs^k h_t(r) \approx\e^{-k r} \bigg(1+\frac{1+r}{t}\bigg)^k h_t(r), \qquad r,t>0.
\]
\end{lemma}
\begin{proof}
We distinguish the cases when $\nu$ is even or odd. In the first case, the result follows by~\eqref{heatkerneleven} and~\cite[Proposition 5.22]{ADY}, according to which for $r>0$ and $p,q\in \N$
\begin{equation}\label{Rpqest}
(-\Rs)^{q} \bigg(-\frac{1}{\sinh r/2} \frac{\partial}{\partial r}\bigg)^{p}\e^{-\frac{r^2}{4t}} \approx \frac{1+r}{t} \bigg(1+ \frac{1+r}{t}\bigg)^{p+q-1}\e^{-(\frac{p}{2}+q)r -\frac{r^2}{4t}}.
\end{equation}
Suppose then that $\nu$ is odd. By~\eqref{heatkernelodd}
\begin{equation}\label{htf0}
\begin{split}
 h_t(r)&= \mathfrak{c}_{0} \pi^{-\frac{1}{2}} \, t^{-\frac{1}{2}}\e^{-\frac{Q^2}{4}t} \int_r^\infty z_r(s) g_0(s) \ ds
    \end{split}
    \end{equation}
where, for any $k \ge 0$, we denote the functions (the function $w_{r}$ will be of use later on)
    \begin{align*}  
       g_k(s)&= \bigg(-\frac{1}{\sinh s} \frac{\partial}{\partial s}\bigg)^{\frac{\nu+1}{2}+k}\bigg(-\frac{1}{\sinh s/2} \frac{\partial}{\partial s}\bigg)^{\frac{\mu}{2}}\e^{-\frac{s^2}{4t}}, \\ 
       z_r(s)&=\frac{\sinh s}{\sqrt{\cosh s-\cosh r}}, \qquad w_r(s)=\frac{1}{\sqrt{\cosh (s)-\cosh r}}.
    \end{align*}  
Given
\[
        f_k(r)= \int_{r}^{\infty} z_{r}(s) g_{k}(s) \, ds  =  \int_0^\infty z_r(r+u)g_k(r+u) \ du,
\]
we shall prove that
\begin{align}\label{rfh}
        -\Rs f_k(r)= f_{k+1}(r).
    \end{align}
Assuming~\eqref{rfh} for a moment, we complete the proof. Indeed, by~\eqref{htf0} and~\eqref{rfh}, one gets
\begin{align}\label{Rkht}
 (-1)^k \Rs^k h_t(r) =  \mathfrak{c}_{0} \pi^{-\frac{1}{2}} \, t^{-\frac{1}{2}}\e^{-\frac{Q^2}{4}t} (-1)^k \Rs^k f_{0}(r) =  \mathfrak{c}_{0} \pi^{-\frac{1}{2}} \, t^{-\frac{1}{2}}\e^{-\frac{Q^2}{4}t}  f_{k}(r).
\end{align}
By~\eqref{Rpqest} and\cite[Proposition 5.26]{ADY}, we obtain
    \begin{align*}
       f_{k}(r)&\approx  \int_r^\infty z_r(s)\frac{1+s}{t}\bigg(1+\frac{1+s}{t}\bigg)^{\frac{\nu+1}{2}+k+\frac{\mu}{2}-1}\e^{-(\frac{\mu}{4}+\frac{\nu+1}{2}+k) s- \frac{s^2}{4t}} \ ds \\ 
        &\approx \frac{(1+r)}{t} \bigg(1+ \frac{1+r}{t}\bigg)^{\frac{\nu+1}{2}+k+\frac{\mu}{2}-\frac{3}{2}}\e^{-(\frac{\mu}{4}+\frac{\nu+1}{2}+k-\frac{1}{2})r}\e^{-\frac{r^2}{4t}},
    \end{align*} so that     \begin{align*}
      t^{-\frac{1}{2}}\e^{-\frac{Q^2}{4}t}  f_{k}(r) \approx \e^{-kr} \bigg(1+ \frac{1+r}{t}\bigg)^k h_t(r),
    \end{align*}
   which completes the proof. We are then left with proving the claim~\eqref{rfh}.
 
 To show this, observe first that
 \[
  \frac{\partial}{\partial r} z_r(r+u)= \frac{\partial}{\partial u} z_r(r+u)  - \sinh r \frac{\partial}{\partial u}   w_r(r+u), \quad \frac{\partial}{\partial r}g_k(r+u)= \frac{\partial}{\partial u}g_k(r+u).
  \]
Then
 \begin{align*}
      \frac{\partial}{\partial r}  f_{k}(r)    
      & =\int_0^\infty \frac{\partial}{\partial r}\bigg( z_r(r+u) g_k(r+u)\bigg) \ du\\ 
      &=  \int_0^\infty \frac{\partial}{\partial r}\big( z_r(r+u)\big) g_k(r+u) \ du + \int_0^\infty z_r(r+u) \frac{\partial}{\partial u}g_k(r+u) \ du \\ 
         &=\int_0^\infty \frac{\partial}{\partial u} \bigg( z_r(r+u)-\sinh r \, w_r(r+u) \bigg) g_k(r+u) \ du \\
         & \qquad + \int_0^\infty z_r(r+u) \frac{\partial}{\partial u} g_k(r+u)\ du.
         \end{align*}
        We now integrate by parts in the first integral, and this gives
         \begin{align*}
  \frac{\partial}{\partial r}  f_{k}(r)           &=\sinh r \int_0^\infty w_r(r+u) \frac{\partial}{\partial u}g_k(r+u) \ du \\ 
         &=-\sinh r \int_0^\infty w_r(r+u)\sinh(r+u) g_{k+1}(r+u) \ du \\ 
         &=-\sinh r  \int_0^\infty z_r(r+u) g_{k+1}(r+u)\ du,
    \end{align*}
     as desired.
\end{proof}
We are now ready to prove Theorem~\ref{teo:main}~(1). Before that, we observe once and for all that for $t,r>0$
\begin{equation}\label{equivaAnk}
 1+ \frac{1}{\sqrt{t}} + \frac{r}{t}  \approx
  \begin{cases}
 1+ \frac{1+r}{t}     \qquad &\mbox{if } r>1,\\
   \max \Big[1, r^2\Big(1+ \frac{1+r}{t}\Big)\Big]^{\frac{1}{2}} \Big(1+\frac{1+r}{t}\Big)^{\frac{1}{2}} &\mbox{if } r\leq 1.
  \end{cases}
\end{equation}
This provides an equivalent form of $\Psi_{k}$ and $\tilde \Psi_{k}$ which, though more involved and apparently overcomplicated, will be more practical now and in the following sections. In particular
\[
\Psi_{k}(r,t)\approx \tilde \Psi_{k}(r,t)\approx \bigg(1+ \frac{1+r}{t}\bigg)^k  \qquad \text{if $r >1$},
\]
while, when $r\leq 1$,
\begin{align*}
\Psi_{k}(r,t)&\approx
 \begin{cases} 
 r \max \Big[1, r^2\Big(1+ \frac{1+r}{t}\Big)\Big]^{\frac{k-1}{2}} \Big(1+\frac{1+r}{t}\Big)^{\frac{k+1}{2}} &\text{if $k$ odd,}\\
 \max\Big[1, r^2\Big(1+\frac{1+r}{t}\Big)\Big]^{\frac{k}{2}}  \Big(1+ \frac{1+r}{t}\Big)^{\frac{k}{2}}  &\text{if $k$ even,}           
 \end{cases}
\end{align*}
and
\begin{align*}
\tilde \Psi_{k}(r,t)& \approx 
 \begin{cases} 
  \max \Big[1, r^2\Big(1+ \frac{1+r}{t}\Big)\Big]^{\frac{k-1}{2}} \Big(1+\frac{1+r}{t}\Big)^{\frac{k-1}{2}}   \Big[r\Big(1+\frac{1+r}{t}\Big)+1 \Big] &\text{if $k$ odd,}\\
 \max\Big[1, r^2\Big(1+\frac{1+r}{t}\Big)\Big]^{\frac{k}{2}}  \Big(1+ \frac{1+r}{t}\Big)^{\frac{k}{2}}  &\text{if $k$ even.}           
 \end{cases}
\end{align*}
The proof of~\eqref{equivaAnk} is elementary and left to the reader.
 \begin{proof}[Proof of Theorem \ref{teo:main}~(1)]
    If $k=1$, the estimate is nothing but~\eqref{boundder}. 
    
    Thus, let $k\geq 3$ be an odd number. By  combining Lemmas~\ref{lem: radial derivates} and~\ref{lem: heatderapprox}, we get
   \begin{equation}\label{h-deriv}
   \begin{split}
  \bigg|\frac{\partial^k}{\partial r^k}  h_t(r)\bigg|
  & \lesssim  h_t(r)  \Bigg[ \sum_{j=1}^{(k+1)/2}   \sum_{1\leq  2\ell+1 \le j} p_{2\ell+1, j- 2\ell-1}  \, \e^{-jr}\bigg(1+\frac{1+r}{t}\bigg)^{j}\\ &\qquad + \sum_{j =(k+3)/2}^{k}  \sum_{0\leq 2\ell \le k-j}  p_{2j-k+2\ell, k-j-2\ell} \, \e^{-jr}\bigg(1+\frac{1+r}{t}\bigg)^{j}  \Bigg].
   \end{split}
   \end{equation}

Suppose now $r >1.$ Then, since $\cosh r \approx \sinh r \asymp \e^r$ and hence $p_{a,b}(r) \approx \e^{(a+b)r}$ for all $a,b\in\N$, we obtain 
\begin{align*}
     \bigg|\frac{\partial^k}{\partial r^k}  h_t(r)\bigg|
     & \lesssim  h_t(r) \sum_{j=1}^{k}   \bigg(1+\frac{1+r}{t}\bigg)^{j}  \lesssim h_t(r)\bigg(1+\frac{1+r}{t}\bigg)^{k},
\end{align*}
as desired.

Suppose now $r<1.$ In this case $\cosh r \approx 1$ while $ \sinh r \asymp r$, whence $p_{a,b}(r) \approx r^{a}$. By \eqref{h-deriv}, we obtain 
\begin{align*}
     \bigg|\frac{\partial^k}{\partial r^k}  h_t(r)\bigg| \lesssim h_t(r) q_{k}(r)
     \end{align*}
     where $q_{k}(r)$ is given by
   \begin{align*}
\sum_{j=1}^{(k+1)/2}   \sum_{1\le 2\ell+1\le j} & r^{2\ell+1}  \bigg(1+\frac{1+r}{t}\bigg)^{j}+ \sum_{j =(k+3)/2}^{k}  \sum_{0\le 2\ell\le k-j}  r^{2j-k+2\ell} \bigg(1+\frac{1+r}{t}\bigg)^{j}  \\
& \lesssim \sum_{j=1}^{(k+1)/2}  r \bigg(1+\frac{1+r}{t}\bigg)^{j}+ \sum_{j =(k+3)/2}^{k}  r^{2j-k} \bigg(1+\frac{1+r}{t}\bigg)^{j},
   \end{align*}
    so that
\begin{equation} \label{qkr}
  q_{k}(r)  \lesssim r\bigg(1+\frac{1+r}{t}\bigg)^{\frac{k+1}{2}} + \begin{cases}
  r^{k}\Big(1+\frac{1+r}{t}\Big)^{k}  &\text{if $r^2\Big(1+\frac{1+r}{t}\Big) \ge 1$,} \\ 
    r^{3}\Big(1+\frac{1+r}{t}\Big)^{\frac{k+3}{2}} &\text{otherwise.}
    \end{cases}
\end{equation}
Depending now on $r^2(1+\frac{1+r}{t})$ being smaller or larger than $1$, there is always a leading term in the sum in~\eqref{qkr}. One gets 
 \begin{align*}
   q_{k}(r)& \lesssim \begin{cases}r^{k}\Big(1+\frac{1+r}{t}\Big)^{k} &\text{if $r^2\Big(1+\frac{1+r}{t}\Big) \ge 1$,} \\ 
    r\Big(1+\frac{1+r}{t}\Big)^{\frac{k+1}{2}} &\text{otherwise.}\end{cases}
    \end{align*}
This proves the statement when $k$ is odd.

The case when $k$ is even is similar, and we omit few details. Let $ k\geq 2$ be an even number. By combining Lemma \ref{lem: radial derivates} with Lemma \ref{lem: heatderapprox}, we get 
\begin{equation*}
\begin{split}
 \bigg|\frac{\partial^k}{\partial r^k} h_t(r)\bigg|& \lesssim   h_t(r) \Bigg[\sum_{j=1}^{k/2} \sum_{0\le 2\ell\le j} p_{2\ell, j-2\ell} \bigg(1+ \frac{1+r}{t}\bigg)^{j}\e^{-jr} \\
 & \qquad \qquad +  \sum_{j=\frac{k}{2}+1}^{k} \bigg(\sum_{0\le 2\ell\le k-j}   p_{2j-k+2\ell, k-j-2\ell} \bigg) \bigg(1+ \frac{1+r}{t}\bigg)^{j}\e^{-j r}\Bigg].
\end{split}
\end{equation*}
If $r > 1$, then
\begin{align*}
   \bigg|\frac{\partial^k}{\partial r^k} h_t(r)\bigg| &\lesssim \bigg( 1+ \frac{1+r}{t}\bigg)^{k}h_t(r),
\end{align*} as desired. If instead $r \le 1$, then by arguing as above one gets 
\begin{align*}
     \bigg|\frac{\partial^k}{\partial r^k} h_t(r)\bigg| &\lesssim h_t(r) \Bigg[\sum_{j=1}^{k/2} \bigg(1+ \frac{1+r}{t}\bigg)^{j} +  \sum_{j=\frac{k}{2}+1}^{k} r^{2j-k} \bigg(1+ \frac{1+r}{t}\bigg)^{j} \Bigg] \\ 
     &\lesssim h_t(r) \times  \begin{cases}
     r^k \Big(1+\frac{1+r}{t}\Big)^k  &\text{if $r^2\Big(1+ \frac{1+r}{t}\Big)> 1$,} \\ 
       \Big(1+ \frac{1+r}{t}\Big)^{k/2}   &\text{if $r^2\Big(1+ \frac{1+r}{t}\Big)\le 1$.}
     \end{cases}
\end{align*}
This concludes the proof.
   \end{proof}

    \subsection{Optimality}
   In this section we prove that in certain regions the estimates of Theorem~\ref{teo:main}~(1) are optimal. We begin by observing that a simplified version of Lemma \ref{lem: radial derivates} is that, for any $k \ge 1$, 
\begin{equation}\label{fjh1}
\frac{\partial^k}{\partial r^k}=\sum_{j=1}^k f_{j,k}(r) \Rs^j,
\end{equation}
    for certain functions $f_{j,k}$ satisfying 
\begin{equation}\label{fjh2}
f_{k,k}(r)=\sinh^k r, \qquad \frac{\partial^m}{\partial r^m} f_{j,k}(r) \geq 0, \qquad 0\le f_{j,k}\lesssim \e^{j} \mbox{ for } r\geq 1,
\end{equation}
   for all $1 \le j \le k$ and $m \in \mathbb N$. 

In order to detect the leading term among those appearing in the radial derivatives of $h_t$ by Lemma~\ref{lem: radial derivates}, especially when $r\leq 1$, we need to show that certain coefficients do not vanish. This leads us to the following proposition. We maintain the notation of~\eqref{fjh1}. 

\begin{proposition}\label{prop: sharp}
    Suppose $k\in \N$. Then there exist constants $c_{k}>0$ and $d_{\ell}=d_{\ell}(k) \ge 0$, $\ell = 1,\dots , \big[ \frac{k}{4}\big]$, such that
       \begin{itemize}
       \item[i)]  if $k$ is odd, then 
    \[
    f_{\frac{k+1}{2},k}(r)=c_{k}\sinh r \cosh^{\frac{k-1}{2}} r+ {\sum_{1\le 2\ell+1\le (k+1)/2} d_{\ell} p_{2\ell+1,\frac{k+1}{2}-2\ell-1}(r)};
    \] 
    \item[ii)] if $k$ is even, then 
    \[
    f_{\frac{k}{2},k}(r)=c_{k}\cosh^{\frac{k}{2}}r {+\sum_{0\le  2\ell \le k/2} d_{\ell} p_{2\ell,\frac{k}{2}-2\ell}(r)}.
    \]
        \end{itemize}
     \end{proposition} 
   
   \begin{proof}
    We proceed by induction. The case $k=1$ is obvious; since 
    \[
    \frac{\partial^2}{\partial r^2}=\cosh r \Rs+\sinh^2 \Rs^2, \qquad \frac{\partial^3}{\partial r^3}=\sinh r+3\sinh r\cosh r \Rs^2+\sinh^3 \Rs^3,
    \] 
one has $f_{1,2}(r)=\cosh r$ and $f_{2,3}(r)=3\sinh r\cosh r$ which satisfy $ii)$ and $i)$. Now assume that 
    $ii)$ holds for some even integer $k \in \mathbb N$. By differentiating
    \begin{align*}
        \frac{\partial}{\partial r}&\sum_{j=1}^k f_{j,k}(r) \Rs^j
        =\sum_{j=1}^{k} \bigg(\frac{d}{d r} f_{j,k}(r) \bigg) \Rs^j + \sum_{j=1}^k f_{j,k}(r)\sinh r \, \Rs^{j+1} \\ 
        &= \bigg(\frac{d }{d r}f_{1,k} (r) \bigg) \Rs+\sum_{j=2}^{k} \bigg( \frac{d }{d r}f_{j,k}(r)+\sinh r f_{j-1,k}(r)\bigg) \Rs^j +\sinh^{k+1} r \, \Rs^{k+1}.
    \end{align*}
    Then, by  Lemma \ref{lem: radial derivates} $ii)$,
\begin{align*}
    f_{\frac{k}{2}+1,k+1}(r) 
    &= \frac{d}{d r}f_{\frac{k}{2}+1,k}(r)+ \sinh r f_{\frac{k}{2},k}(r)  \\ 
    &=\sum_{1\leq 2\ell+1\le \frac{k}{2}+1} \tilde{d}_{\ell} p_{2\ell+1,\frac{k}{2}-2\ell}(r)+\sum_{3\leq 2\ell+3\le \frac{k}{2}+1} \tilde{d}_{\ell} p_{2\ell+3,\frac{k}{2}-2\ell-2}(r)\\&\qquad +c_{k}\sinh r\cosh^{\frac{k}{2}}r+ \sum_{0 \leq 2\ell+1 \le \frac{k}{2}+1} d_{\ell} p_{2\ell+1,\frac{k}{2}-2\ell}(r),
\end{align*}
    from which the conclusion follows. The proof of $ii)$ follows the same reasoning; if $i)$ holds for some odd integer $k$, then by Lemma~\ref{lem: radial derivates} $i)$
    \begin{align*}
    f_{\frac{k+1}{2},k+1}(r)
   &=\frac{d}{d r}f_{\frac{k+1}{2},k}(r)+ \sinh r f_{\frac{k-1}{2},k}(r)\\ 
    &=c_{k} \cosh ^{\frac{k+1}{2}}(r)+c_{k}\frac{k-1}{2}\sinh^2 r\cosh^{\frac{k-3}{2}}r\\& \qquad + \sum_{2 \leq 2\ell \le \frac{k-1}{2}} d_{\ell} p_{2\ell,\frac{k+1}{2}-2\ell}(r)+ \sum_{0\le 2\ell \le \frac{k-3}{2}} d_{\ell} p_{2(\ell+1),\frac{k+1}{2}-2(\ell-1)}(r)\\ & \qquad +\sum_{2 \leq 2(\ell+1)\le \frac{k+1}{2}}d_\ell p_{2(\ell+1),(k+1)/2-2(\ell+1)}(r),
    \end{align*} 
    as desired.
    \end{proof}

We are now ready to prove that, in certain regimes, Theorem~\ref{teo:main}~(1) is optimal.

 \begin{proposition}\label{prop: sharpness} Suppose $k\in \N$ and $\alpha \in [1,2)$. There exists a constant $\gamma_{k,\alpha}>0$ such that, if $(1+r)/t \ge \gamma_{k,\alpha}$, then
\begin{equation}\label{stimabasso}
  \bigg| \frac{\partial^k}{\partial r^k} h_t(r)\bigg| \gtrsim \Psi_{k}(r,t) h_{t}(r)
\end{equation}
 when either  $ r>1$ or $r^{\alpha}\big(1+\frac{1+r}{t}\big)<1$.
  \end{proposition}
 \begin{proof}
In the proof we shall make heavy use of Lemma~\ref{lem: heatderapprox} without further mention. We first prove~\eqref{stimabasso} when $r>1$. By~\eqref{fjh1} and~\eqref{fjh2}
    \begin{align*} 
        \bigg| \frac{\partial^k}{\partial r^k} h_t(r)\bigg| 
        &\ge \sinh^k r |\Rs^k h_t(r)|-\sum_{j=1}^{k-1} |f_{j,k} \Rs^j h_t(r)| \\
        &\ge c\bigg(1 + \frac{1+r}{t}\bigg)^k h_t(r)-\sum_{j=1}^{k-1} |f_{j,k} \Rs^j h_t(r)|
    \end{align*}
    for a certain $c>0$.
Since moreover
\[
|f_{j,k}(r)\Rs^j h_t(r)| \lesssim \bigg(1 + \frac{1+r}{t}\bigg)^j h_t(r), \qquad 1 \le j \le k-1
\]
by~\eqref{fjh2},~\eqref{stimabasso} follows if $\frac{r+1}{t}$ is large enough.

Assume now that $r^{\alpha}\big(1+ \frac{1+r}{t}\big)<1$ for $\alpha \in [1,2)$. Since in particular $r<1$, one gets $r^{2}\big(1+ \frac{1+r}{t}\big)<1$, whence
\begin{align*}
   \Psi_{k}(r,t) \approx h_t(r) \times \begin{cases} r\big(1 + \frac{1+r}{t}\big)^{\frac{k+1}{2}} &\text{if $k$ is odd}, \\  \big(1 + \frac{1+r}{t}\big)^{\frac{k}{2}} &\text{if $k$ is even}.
        \end{cases}
\end{align*}
If $k$ is even, then by Proposition \ref{prop: sharp} one has that
\begin{align*}
    |f_{\frac{k}{2},k}(r)\Rs^{\frac{k}{2}} h_t(r)| \approx \bigg(1 + \frac{1+r}{t}\bigg)^{\frac{k}{2}} h_t(r),
\end{align*} while, by Lemma~\ref{lem: radial derivates} $ii)$ 
     \begin{align*}
         |f_{j,k}(r)\Rs^{j} h_t(r)| &\lesssim  r^{2j-k} |\Rs^{j} h_t(r)| \approx r^{2j-k} \bigg(1 + \frac{1+r}{t}\bigg)^{j} h_{t}(r)  &&\qquad \tfrac{k}{2}+1\le j \le k,
              \end{align*} 
         and
         \begin{align*}    
         |f_{j,k}(r)\Rs^{j} h_t(r)| &\lesssim   |\Rs^{j} h_t(r)| \approx \bigg(1 + \frac{1+r}{t}\bigg)^{j}h_{t}(r) \ \ \ \ &&\qquad  1\le j \leq \tfrac{k}{2}-1.
     \end{align*} 
     Therefore, for $1 \le j \le k$ with $j \ne k/2$,
     \begin{align*}
           |f_{j,k}(r)\Rs^{j} h_t(r)| &\lesssim \bigg(1 + \frac{1+r}{t}\bigg)^{\frac{k}{2}-\frac{2-\alpha}{\alpha}}h_{t}(r),
     \end{align*}
     from which the desired estimate follows again provided $\frac{r+1}{t}$ is large enough.
     
The case $k$ is odd is completely analogous. By Proposition \ref{prop: sharp}
        \begin{align*}
         |f_{\frac{k+1}{2},k}(r)\Rs^{\frac{k+1}{2}} h_t(r)| \approx r\bigg (1 + \frac{1+r}{t}\bigg)^{\frac{k+1}{2}} h_t(r).
     \end{align*}  
     For any $1 \le j < (k+1)/2$, by Lemma \ref{lem: radial derivates} $i)$ we get 
     \begin{align*}
          |f_{j,k}(r)\Rs^{j} h_t(r)| \lesssim r\bigg (1 + \frac{1+r}{t}\bigg)^{\frac{k-1}{2}} h_t(r),
     \end{align*}  while for any $(k+1)/2 < j \le k $
     \begin{align*}
           |f_{j,k}(r)\Rs^{j} h_t(r)|
           &\lesssim r^{2j-k}\bigg(1+\frac{1+r}{t}\bigg)^{j}h_{t}(r)\le r\bigg (1 + \frac{1+r}{t}\bigg)^{\frac{k+1}{2}-\frac{2-\alpha}{\alpha}} h_t(r),
     \end{align*} 
which imply the desired estimate also in the odd case.
     \end{proof}

\section{Derivatives of the Riemannian distance}\label{sec:distance}

It is well known, cf.~\cite[(2.18)]{ADY}, that
\[
\cosh ^2\bigg(\frac{r(x,z,a)}{2}\bigg)=\bigg(\frac{a^{1/2}+a^{-1/2}}{2}+\frac{1}{8}\,a^{-1/2}|x|^2\bigg)^2+\frac{1}{4}\,a^{-1}|z|^2,
\]
or equivalently (we often do not stress the dependence on $(x,z,a)$)
\begin{equation}\label{coshd}
\cosh r =-1+ 2\bigg(\frac{a^{1/2}+a^{-1/2}}{2}+\frac{1}{8}\,a^{-1/2}|x|^2\bigg)^2+\frac{1}{2}\,a^{-1}|z|^2.
\end{equation}
At a local scale, the Riemannian distance $r$ behaves like the Euclidean distance. This is shown in the following simple lemma.
\begin{lemma}
If $r\leq 1$, then $r\approx |a-1| +|x| + |z|$.
\end{lemma}

\begin{proof}
An elementary computation shows that
\[
\cosh r-1 = \frac{1}{2a} \bigg( (a-1)^2 + \frac{1}{16}|x|^4 +  \frac{a}{2}(1+a^{-1})|x|^{2} + |z|^2 \bigg).
\]
Since $\cosh r -1 \approx  \frac{r^2}{2}$ for $r\leq 1$, while $a \approx 1$ and $|x| \lesssim 1$  if $r\leq 1$, the conclusion follows.
\end{proof}

The aim of this section is to prove Theorem~\ref{teo:derd}. The case $k=1$ is well known, since $|\nabla r| \leq 1$. In order to obtain the other cases, the first ingredients will be the classical and a multivariate versions of Faà di Bruno's formula.

In its classical form, given a vector field $\bX \in \mathfrak{s}$, an integer $\ell\in\N$, and two compatible functions $g\colon [0,\infty) \to \R$ and $w\colon S \to [0,\infty)$, such formula states that
\begin{align}\label{Faa}
    \bX^{\ell} (g\circ w) = \ell! \sum_{j=1}^\ell  \frac{1}{j!}\bigg(\frac{\partial^{j}}{\partial r^{j}} g\bigg)(w) \sum_{m_1+\cdots +m_j=\ell} \frac{\bX^{m_1}w}{m_{1}!} \cdots \frac{\bX^{m_j}w}{m_j!},
\end{align} 
where the inner sum is meant to run over all possible integers $m_1, \dots, m_j \geq 1$ whose sum is $\ell$ (this condition will always be implicit in the above notation).  A multivariate version turns out to be much more involved, but the weaker form which we now prove will be enough to us.

Here and in the following, it will be convenient to identify a multi-index $J$ with the ordered set of its entries (where repeated elements are considered to be different). Given two multi-indices $J,I$, we write $J\subseteq I$ if there exists an injective ``order preserving'' map $\sigma \colon J \to I$, i.e.\ such that, if $j_{1} \leq j_{2}$ and $i_{1},i_{2}$ are such that $\sigma(J_{j_{1}} ) = I_{i_{1}}$ and $\sigma(J_{j_{2}} ) = I_{i_{2}}$, then $i_{1}\leq i_{2}$.

\begin{lemma}\label{Faamisto}
Suppose $b\in \N$ and $J \in \{ 0,\dots, n-1\}^{b}$. There exist functions $\sigma_{j,J}$ such that, for all smooth radial functions $g$ on $S$,
\begin{equation}\label{sigmajJ}
\bX_J g = \sum_{j=1}^{b} g^{(j)}  \sigma_{j,J}.
\end{equation}
In particular, $\sigma_{1,J} = \bX_{J}r$ and $\sigma_{b,J} = \prod_{j=1}^{b}\bX_{J_{j}}r$.
\end{lemma}
\begin{proof}
Let $\mathcal{J}(J)$ be the family of all partitions of $J$, namely the set of vectors of multi-indices $ \vec{J}=\{ J^{(1)}, \dots, J^{(m)} \}$ for some $m\in \{1,\dots, b\}$ with the property that $J^{(\ell)} \subseteq J$ for all $\ell=1,\dots, m$ and $J^{(\ell)}\cap J^{(h)}= \emptyset$ for all $h\neq \ell$, and that $\sum_{\ell=1}^{m}|J^{(\ell)}|=|J| $. By induction on $b\in \N$, one proves that
\[
\bX_J g = \sum_{j=1}^{b} g^{(j)}  \sigma_{j,J}
\]
where
\begin{equation}\label{sigmajJformula}
\sigma_{j,J}(r) = \sum_{\vec{J}=\{ J^{(1)}, \dots, J^{(j)} \} \in \mathcal{J}(J)} \prod_{\ell=1}^{j} \epsilon_{\ell, \vec{J}}\,  \bX_{J^{(\ell)}}r
\end{equation}
for certains constants $ \epsilon_{\ell, \vec{J}}\in \{0,1\}$, and in particular $\epsilon_{1, \vec{J}} = \epsilon_{b, \vec{J}}=1$.  
\end{proof}

If we apply Lemma~\ref{Faamisto}  to $g(r)=\cosh r$, we get the following. 
\begin{lemma}\label{weakFaa}
For all $m\in \N$, $m\geq 2$, and $J \in \{0,\dots, n-1\}^{m}$,
\begin{equation}\label{generale}
   \bX_{J} (\cosh r)   =   (\bX_{J}r ) \sinh{r} + f_{J}(r)\cosh r+g_{J}(r)\sinh{r},
\end{equation}
where $f_{J}$ is a linear combination of functions of the form  
\[ \bX_{I_{1}}r\cdot \bX_{I_{2}}r \cdots  \bX_{I_{p}}r, \qquad \mbox{ $p$ even}, \qquad  |I_{1}| + \dots +|I_{p}| \le m-1,
\]
and $g_J$ is a linear combination of 
\[
\bX_{\tilde I_{1}}r \cdot \bX_{\tilde I_{2}}r \cdots \bX_{\tilde I_{p}}r, \qquad \mbox{ $p$ odd}, \qquad |\tilde I_{1}| + \dots +|\tilde I_{p}| \le m-2
\]
($g_{J}=0$ if $m=2$). Moreover, $I_{{k}}, \tilde I_{k}\subseteq J$ for all $k$.
\end{lemma}
\begin{proof}
We recall that radial derivatives of odd (even) order of $\cosh r$ are equal to $\sinh r$ ($\cosh r$). Then, \eqref{generale} immediately follows by \eqref{sigmajJ} and \eqref{sigmajJformula}.
\end{proof}

The key to prove Theorem~\ref{teo:derd} is the following lemma, whose proof we postpone for a moment. Right after the statement, we show that Theorem~\ref{teo:derd} follows at once.
\begin{lemma}\label{lem: dercosh}
For all $m\in \N$ and $J \in \{0,\dots, n-1\}^{m}$
\[
   |\bX_{J} (\cosh r)|  \lesssim \cosh r.
\]
\end{lemma}
\begin{proof}[Proof of Theorem~\ref{teo:derd}]
By Lemma \ref{weakFaa}, for $J \in \{0,\dots, n-1\}^k$, $k \ge 2$
\[
    |\bX_J r| \lesssim  \frac{1}{\sinh r} ( |\bX_J(\cosh r)|+|f_J(r)|\cosh r +|g_J(r)| \sinh r),
\]
whence by Lemma~\ref{lem: dercosh} 
\begin{align}\label{dercosh}
     |\bX_J r | \lesssim \coth r(1+|f_J(r)|)+|g_J(r)|.
\end{align} 
If $k=2$, then $f_J(r)$ is a product of first derivatives of the distance, whence $|f_J(r)| \lesssim 1$, while $g_{J}=0$. Then by~\eqref{dercosh} 
\begin{align*}
    |\bX_J(r)| \lesssim \coth r \approx \frac{1+r}{r}.
\end{align*} 
Assume now that for $k\geq 2$, $j=1,\dots, k$ and  $\tilde{J} \in \{0,\dots, n-1\}^{j}$
\[
|\bX_{\tilde J} r | \lesssim  \left(\frac{1+r}{r}\right)^{j-1},
\]
and pick $J \in \{0,\dots, n-1\}^{k+1}$. By~\eqref{dercosh}, one gets 
\begin{align*}
     |\bX_J(r)| &\lesssim \coth r (1+|f_J(r)|)+|g_J(r)| \\
     & \lesssim  \coth r \bigg(1+ \left(\frac{1+r}{r}\right)^{k-1}\bigg)+ \left(\frac{1+r}{r}\right)^{k-2} \lesssim \left( \frac{1+r}{r}\right)^{k},
\end{align*} as desired.
\end{proof}
We are then left with proving Lemma \ref{lem: dercosh}. 
\begin{proof}[Proof of Lemma \ref{lem: dercosh}] 
We begin by observing that the statement holds when $m=1$, as for $\ell=0, \dots, n-1$
\[
|\bX_{\ell} \cosh r | = |\sinh r \bX_{\ell} r| \leq \sinh r \leq \cosh r.
\]
Moreover, the statement also holds when $m$ is arbitrary and $0\notin J$, namely $J\in \{ 1, \dots, n-1\}^{m}$. Indeed, $\cosh r$ is a polynomial of degree $4$ in $x$ and degree $2$ in $z$, and one sees that
\begin{align*}
    \bX_J \cosh r=0 \quad \mbox{for all } J \in \{ 1, \dots, n-1\}^{m}, \qquad m\geq 5. 
    \end{align*}
The fact that $|\bX_J \cosh r|\lesssim \cosh r$ for $0\notin J$ and $2\leq m\leq 4$ can be checked via elementary (though tedious) computations.

Then, we consider the case when all the derivatives are along $\bX_{0}$, i.e.\ $J=\{0\}^{m}$. The key observation in this case is that by~\eqref{coshd}  
\[
\bX_0^2(\cosh r)= \cosh r-\frac{|x|^2}{4}.
\]
Therefore, for all $k\in\N$, $k\geq 1$,
\begin{align}\label{derivative0coshd}
 \bX_0^{2k} \cosh r= \cosh r- \frac{|x|^2}{4}, \qquad   \bX_0^{2k +1} \cosh r=\bX_0 \cosh r.
\end{align}
Since $|x|^2 \lesssim \cosh r$ and
\begin{align}\label{X0cosh}
    \bX_0(\cosh r)=\frac{1}{2} a-2\bigg(\frac{a^{-1/2}}{2}+\frac{1}{8}\,a^{-1/2}|x|^2\bigg)^2-\frac{1}{2}\,a^{-1}|z|^2,
\end{align} 
 we easily get that 
\[
| \bX_0^{2k} \cosh r| \lesssim \cosh r, \qquad  |\bX_0^{2k+1} \cosh r| \lesssim \cosh r.
\]
We are then left with proving that the desired bound holds for arbitrary mixed derivatives. We observe that by the commutation relations~\eqref{commcampi}, for $1\leq j\leq n-1$,
\[
[\bX_0, \bX_j] = c_j\bX_j, \quad \mbox{equivalently} \quad \bX_0 \bX_j = \bX_j \bX_0 + c_j\bX_j,
\]
with $c_j \in \{1/2,1\}$ (precisely, $c_j = 1/2$ if $j=1,\dots,\mu$, while $c_j = 1$ if $j=\mu+1, \dots, n-1$). Therefore, for all $J \in \{0,\dots, n-1\}^m$
\begin{align}\label{commutation2}
\bX_J = \sum_{|I'|<m, \, 0\notin I'} c_{I'}\bX_{I'} + \bX_I \bX_0^p
\end{align}
for some $0\leq p\leq m$, $I \in \{1,\dots, n-1\}^{m-p}$ and coefficients $c_{I'}$ (which might well be zero). In particular, $p$ is the number of occurrences of $0$ in $J$, and we can assume $0<p<m$; otherwise, the statement follows by previous cases. 

Suppose $J \in \{0,\dots, n-1\}^{m}$ and $0<p<m$. By~\eqref{commutation2} and the preceding cases
\begin{align*}
    |\bX_J(\cosh r)|&
     \lesssim \cosh r + |\bX_I \bX_0^p \cosh r|
     \\  &=\cosh r+ \begin{cases} \big|\bX_I\big(\cosh r-\frac{|x|^2}{4}\big)\big| &\text{if $p$ is even,} \\
    |\bX_I\bX_0 \cosh r| &\text{if $p$ is odd.}
    \end{cases}
\end{align*} 
It remains to realize now that for all $I \in \{1,\dots, n-1\}^{j}$, $j\in\N$,
\[
\bigg|\bX_I\bigg(\cosh r-\frac{|x|^2}{4}\bigg)\bigg| \lesssim \cosh r +|\bX_I (|x|^{2})| \lesssim \cosh r,
\]
as well as, by~\eqref{X0cosh}, that
\begin{align*}
     |\bX_I \bX_0(\cosh r)| = 0 \quad \mbox{if} \quad |I|\geq 5,  \qquad    |\bX_I \bX_0(\cosh r)| \lesssim \cosh r \quad \mbox{if } 1\leq |I|\leq 4.
\end{align*} 
The proof is then complete.
\end{proof}

We conclude this section by stating a corollary of Theorem~\ref{teo:derd} and Lemma~\ref{Faamisto}, whose notation we maintain, and by observing that the estimates at infinity of the derivatives of the distance of Theorem~\ref{teo:derd} are sharp; i.e., the derivatives of $r$ can admit a non-zero limit at infinity. 
\begin{corollary}
Suppose $b\in \N$ and $J \in \{0,\dots, n-1\}^b$. Then
\begin{align}\label{estimatesigmajJ}
|\sigma_{j,J}(r)| \lesssim \bigg( \frac{1+r}{r}\bigg)^{b-j}, \qquad r>0, \; j=1,\dots, b.
\end{align}  
\end{corollary}
\begin{proof}
Just combine~\eqref{sigmajJformula} with Theorem~\ref{teo:derd} to get
\begin{align*}
    |\sigma_{j,J}(r)| \lesssim \sum_{\vec{J}=\{ J^{(1)}, \dots, J^{(j)} \} \in \mathcal{J}(J)} \bigg( \frac{1+r}{r}\bigg)^{ \sum_{\ell=1}^{j} (|J^{(\ell)}| -1)} \lesssim \bigg( \frac{1+r}{r}\bigg)^{b-j},
\end{align*} as desired.
\end{proof}

\begin{proposition}\label{sharpinf}
For all $k\in \N$ there exists a smooth curve $\gamma_{k} \colon (0,\infty) \to S$ and a constant $b_{k}\neq 0$ such that $\lim_{a\to \infty} r(\gamma_{k}(a))=+\infty$ and $ \lim_{a\to \infty}(\bX_{0}^{k}r)(\gamma_{k}(a)) = b_{k}$.
\end{proposition}

\begin{proof}
Suppose $y\in (-1,1)$. Then, the curve
\[
\eta_y (a) = \bigg(0, \sqrt{\frac{1-y}{1+y}}a u_{1}, a\bigg), \qquad a>0,
\]
satisfies $\lim_{a\to \infty} r( \eta_{y}(a))=+\infty$ (observe that $\lim_{a\to \infty} \cosh r(\eta_{y}(a)) = + \infty$) and 
\[
(\bX_{0}r)(  \eta_{y}(a)) = \frac{(\bX_{0}\cosh r)( \eta_{y}(a))}{\sinh r( \eta_{y}(a))} \to y, \qquad a\to \infty.
\]
Therefore, the statement holds for $k=1$. We now claim that for all $k\geq 2$ the limit
\begin{equation}\label{pollimit}
P_{k}(y) \coloneqq \lim_{a\to \infty} (\bX_{0}^{k}r)(\eta_{y}(a)), \qquad y\in (-1,1)
\end{equation}
exists and is a non-null polynomial of degree $k$ in $y$. Therefore, for all $k\in \N$ there is $y_{k}\in (-1,1)$ such that $P_{k}(y_{k}) \neq 0$, and the choice $\gamma_{k}= \eta_{y_{k}}$ gives the desired curve.

We proceed by induction, and  show that~\eqref{pollimit} exists and that the coefficient of its highest degree term is $(-1)^{k-1}(k-1)!$. For $k=1$ this has just been proved. Assume that this holds up to $k-1$ for some $k\geq 2$. From the Faà di Bruno formula~\eqref{Faa},
\[
\bX_{0}^{k}r = \frac{\bX_{0}^{k}\cosh r}{\sinh r} - k! \sum_{j=2}^k  \frac{1}{j!} \frac{(\cosh r)^{(j)}}{\sinh r}  \sum_{m_1+\cdots +m_j=k} \frac{\bX_{0}^{m_1} r}{m_{1}!} \cdots \frac{\bX_{0}^{m_j}r}{m_j!}.
\]
Since the outer sum runs over $j\geq 2$, all the $m_{\ell}$'s in the inner sum are $\leq k-1$, and the inductive assumption ensures that the limit~\eqref{pollimit} exists. By~\eqref{derivative0coshd}, moreover,
\[
\lim_{a\to \infty}   \frac{\bX_{0}^{k}\cosh r}{\sinh r} (\eta_{y}(a)) = y \quad \mbox{if $k$ odd},\qquad \lim_{a\to \infty}   \frac{\bX_{0}^{k}\cosh r}{\sinh r} (\eta_{y}(a)) = 1 \quad \mbox{if $k$ even}.
\] 
By this and the inductive assumption, $P_{k}$ is a polynomial of degree $k$ and the coefficient of its term of highest degree is
\[
- k! \sum_{j=2}^k  \frac{1}{j!} \sum_{m_1+\cdots +m_j=k} \frac{(-1)^{m_{1}}(m_{1}-1)!}{m_{1}!} \cdots \frac{(-1)^{m_{j}}(m_{j}-1)!}{m_{j}!} .
\]
It remains to prove that the above equals $(-1)^{k}(k-1)!$, but this is the same as proving that
\[
M_{k}\coloneqq k!\sum_{j=1}^k  \frac{1}{j!} \sum_{m_1+\cdots +m_j=k} \frac{(-1)^{m_{1}}(m_{1}-1)!}{m_{1}!} \cdots \frac{(-1)^{m_{j}}(m_{j}-1)!}{m_{j}!} =0.
\]
This immediately follows by Faà di Bruno's formula, as $M_{k}=(\e^{\ln x})^{(k)}(1)=0$.
\end{proof}

\section{Space and time derivatives of $h_{t}$}\label{sec: 4}
The aim of this section is to prove Theorem~\ref{teo:main}~(2). As a first step, we shall prove the estimates for the space derivatives, namely~(2) when $m=0$. The case $m>0$ will be treated afterwards. 
\subsection{Space derivatives}
We begin by stating the aforementioned special case of Theorem~\ref{teo:main}~(2), which reads as follows.
\begin{proposition}\label{prop:spaceder}
Suppose $k\in \N$ and $J\in \{0,\dots, n-1\}^{k}$. Then
\[
| \bX_{J} h_{t}(r)| \lesssim \tilde \Psi_{k}(r,t) h_{t}(r), \qquad \forall \, t,r>0,
\]
and in the particular case when $J = \{j\}^{k}$ for some $j\in\{0,\dots, n-1\}$, 
\[
|  \bX_j^{k} h_{t}(r)| \lesssim  \Psi_{k}(r,t) h_{t}(r), \qquad \forall \, t,r>0.
\]
\end{proposition}
As already mentioned, one cannot replace $\tilde \Psi$ by $\Psi$ in the first formula above. We shall go into details later on.

The estimates in Proposition~\ref{prop:spaceder} for large radii are easy to get as they immediately follow from the Faà di Bruno formula, in particular Lemma~\ref{Faamisto} and~\eqref{estimatesigmajJ}.
\begin{proposition}\label{prop:dermag1}
Suppose $k\in \N$ and $J\in \{0,\dots, n-1\}^{k}$. Then 
\[
| \bX_{J} h_{t}(r)| \lesssim   \left(1+ \frac{1+r}{t}\right)^{k}h_{t}(r) , \qquad r\geq 1,\,  t>0.
\]
\end{proposition}
\begin{proof}
By combining Lemma~\ref{Faamisto}, Theorem~\ref{teo:main}~(1) and~\eqref{estimatesigmajJ} one gets, for $r\geq 1$,
\[
|\bX_J h_{t}(r)| \lesssim \sum_{j=1}^{k} |h_{t}^{(j)}|  |\sigma_{j,J}| \lesssim h_{t}(r) \sum_{j=1}^{k}\left(1+ \frac{1+r}{t}\right)^{j} \lesssim  \left(1+ \frac{1+r}{t}\right)^{k}h_{t}(r)
\]
which is the claimed estimate.
\end{proof}
The simple argument contained in Proposition~\ref{prop:dermag1} above does not work when $r\leq 1$. Since the derivatives of the distance are singular at $0$, indeed, the same reasoning leads to 
\[
|\bX_J h_{t}(r)| \lesssim h_{t}(r) \sum_{j=1}^{k} \Psi_{j}(r,t) r^{j-k},
\]
namely to an estimate which is in turn singular at $0$. Since $h_{t}$ is smooth, such a bound is too rough; and it suggests that some cancellation should occur among a number of terms. This is what we show in Proposition~\ref{propodd} below, before which we need a certain amount of preliminaries.

We begin by introducing some functions which will be used throughout the remainder of the section. For $b, q\in \N$, write
\[
\Phi_{0,2q,t}= h_{t}^{(2q+1)}, \qquad \Phi_{2b,2q,t} (r)=  \sum_{\ell=0}^{2b-1} \frac{(-1)^{\ell} r^{\ell}}{\ell!}h_{t}^{(2q+\ell+1)}(r), \qquad b\geq 1. 
\]

\begin{lemma}\label{lemmaphi}
 Suppose $b,q\in \N$. Then for all $t>0$ and $r\leq 1$
 \[
|\Phi_{2b,2q,t}(r)|\lesssim r^{2b} \Psi_{2b+2q+1}(r,t) h_{t}(r).
\]
\end{lemma}
\begin{proof}
The case $b=0$ is part of Theorem~\ref{teo:main}~(1). If $b \geq 1$, one proves by induction that 
\begin{align}\label{per parti}
     \Phi_{2b,2q,t}(r)=-\frac{1}{(2b-1)!}\int_0^r s^{2b-1}h_t^{(2b+2q+1)}(s) \, \dd s.
\end{align} 
Suppose $t>0$ and $r\leq 1$. We distinguish two cases, either \begin{align}\label{caso 1}
    r^2\bigg(1 + \frac{1+r}{t}\bigg)\le 1
\end{align} or 
\begin{align}\label{caso 2}
    r^2\bigg(1 + \frac{1+r}{t}\bigg)>1. 
\end{align}
Assume~\eqref{caso 1}, and observe that this is equivalent to $\frac{r^2}{t} \le 1-r$. Therefore, in this case $ \e^{-\frac{r^2}{4t}}  \approx 1$, and by~\eqref{boundsht}, for $r\leq 1$
\begin{align*}
     h_t(0) \approx h_t(r).
\end{align*} 
By \eqref{per parti} and Theorem~\ref{teo:main}~(1)
\begin{align*}
    |\Phi_{2b,2q,t}(r)|&\lesssim \int_0^r s^{2b-1}|h_t^{(2b+2q+1)}(s)| \ ds \\ 
    &\lesssim \int_0^r s^{2b-1}s \bigg(1+\frac{1+s}{t}\bigg)^{b+q+1}h_t(s) \ ds \\ 
    &\lesssim  r^{2b+1}\bigg(1+\frac{1+r}{t}\bigg)^{b+q+1}h_t(r),
\end{align*} 
as desired. Now assume \eqref{caso 2}. By Theorem~\ref{teo:main}~(1)
\begin{align*}
   |\Phi_{2b,2q,t}(r)| &\lesssim \sum_{\ell=0}^{2b-1} r^{\ell+2q+\ell+1}\bigg(1+ \frac{1+r}{t}\bigg)^{2q+\ell+1}h_t(r)\\ 
   &=r^{2q+1}\bigg(1+ \frac{1+r}{t}\bigg)^{2q+1} h_t(r)\sum_{\ell=0}^{2b-1} r^{2\ell}\bigg(1+ \frac{1+r}{t}\bigg)^{\ell}\\ 
   &\approx r^{2q+4b-1}\bigg(1+ \frac{1+r}{t}\bigg)^{2q+2b} h_t(r) \\ 
   &<r^{2q+4b+1}\bigg(1+ \frac{1+r}{t}\bigg)^{2q+2b+1} h_t(r),
\end{align*} as desired. In the  last inequality we have used \eqref{caso 2}. This concludes the proof.
\end{proof}

For $k\in \N$, $k\geq 1$, and $J\in \{ 0, \dots, n-1\}^{k}$, define now the functions
\[
\Upsilon_{2j+1,J}= \sigma_{2j+1,J}  - \sum_{\ell=0}^{j-1} \frac{r^{2j-2\ell} }{(2j-2\ell)!} \Upsilon_{2\ell+1,J}, \qquad j =0, \dots, \big[\tfrac{k-1}{2}\big]-1
\]
($\sigma_{j,J}$ are those of~\eqref{sigmajJ}), as well as
\[
\Upsilon_{k-1,J}= \sigma_{k-1,J} \qquad \mbox{if $k$ is even}, \qquad \Upsilon_{k,J}= \sigma_{k,J} \qquad \mbox{if $k$ is odd}.
\]
Then, define also
\[
\Xi_{2j,J} = \sigma_{2j,J} + \sum_{\ell=0}^{j-1}\frac{r^{2j-2\ell-1}}{(2j-2\ell-1)!}\Upsilon_{2\ell+1,J}, \qquad j= 1, \dots,  \big[\tfrac{k-1}{2}\big].
\]
Observe that all the functions $\Upsilon_{2j+1,J}$'s and $\Xi_{2j,J}$'s implicitly depend on $k$, since so does $J$. Therefore, their definitions above are well posed.

\begin{lemma}\label{xismooth}
 Suppose $k\in \N$  and $J\in \{ 0, \dots, n-1\}^{k}$. Then the functions $\Xi_{2j,J}$, $j=1,\dots, \big[\tfrac{k-1}{2}\big]$, are smooth.
\end{lemma}

\begin{proof}
We claim that for $j=1,\dots, \big[\tfrac{k-1}{2}\big]$
\begin{equation}\label{eqXi0}
\Xi_{2j,J} = \frac{1}{(2j)!}\bX_{J}(r^{2j}) - \sum_{\ell=1}^{j-1} \frac{1}{(2j-2\ell)!} r^{2j-2\ell}\Xi_{2\ell,J},
\end{equation}
which amounts to
\begin{equation}\label{eqXi}
 \frac{1}{(2j)!}\bX_{J}(r^{2j}) = \sum_{\ell=1}^{j} \frac{ r^{2j-2\ell}}{(2j-2\ell)!}\Xi_{2\ell,J}.
\end{equation}
Since by~\eqref{sigmajJ}
\[
\Xi_{2,J} =  \sigma_{2,J} + r \sigma_{1,J} = \frac{1}{2} \bX_J(r^{2}),
\]
whence $\Xi_{2,J}$ is smooth by the smoothness of $r^{2}$, the smoothness of $\Xi_{2j,J}$ for $j\geq 2$ will follow by~\eqref{eqXi0}.

We now prove~\eqref{eqXi}. First, we observe that by definition of $\Upsilon_{2j+1,J}$
\begin{equation}\label{sigmaUps}
 \sigma_{2\ell+1,J} =  \sum_{p=0}^{\ell} \frac{ r^{2\ell-2p}}{(2\ell-2p)!} \Upsilon_{2p+1,J},
\end{equation}
so that, since $k>2j$ and by Lemma \ref{Faamisto},
\begin{align*}
\frac{1}{(2j)!}\bX_{J}(r^{2j}) 
&= \sum_{\ell=1}^{2j} \frac{r^{2j-\ell} }{(2j-\ell)!} \sigma_{\ell,J}\\
& = \sum_{\ell=1}^{j} \frac{r^{2j-2\ell}}{(2j-2\ell)!}  \sigma_{2\ell,J}  + \sum_{\ell=1}^{j} \frac{r^{2j-2\ell+1}}{(2j-2\ell+1)!}  \sigma_{2\ell-1,J} \end{align*}
where we have just split the sum into odd and even indices respectively. We now make use of~\eqref{sigmaUps} to get
\begin{align*}
\frac{1}{(2j)!}\bX_{J}(r^{2j}) 
& = \sum_{\ell=1}^{j} \frac{r^{2j-2\ell}}{(2j-2\ell)!}  \sigma_{2\ell,J}  + \sum_{\ell=1}^{j} \frac{r^{2j-2\ell}}{(2j-2\ell+1)!}\sum_{p=0}^{\ell-1} \frac{r^{2\ell-2p-1} }{(2\ell-2p-2)!} \Upsilon_{2p+1,J}\\
& = \sum_{\ell=1}^{j} \frac{r^{2j-2\ell}}{(2j-2\ell)!} \bigg\{  \sigma_{2\ell,J}  + \frac{1}{(2j-2\ell+1)}\sum_{p=0}^{\ell-1} \frac{r^{2\ell-2p-1} }{(2\ell-2p-2)!} \Upsilon_{2p+1,J} \bigg\}.
\end{align*}
We now consider the term in brackets, which equals
\begin{align*}
& \sigma_{2\ell,J} + \sum_{p=0}^{\ell-1}\frac{r^{2\ell-2p-1}}{(2\ell-2p-1)!}\Upsilon_{2p+1,J} \\
 &\quad \qquad  - \sum_{p=0}^{\ell-1}r^{2\ell-2p-1}\Upsilon_{2p+1,J} \bigg(\frac{1}{(2\ell-2p-1)!} - \frac{1}{(2j-2\ell+1)(2\ell-2p-2)!} \bigg) \\
 & \;\; = \Xi_{2\ell,J} - \sum_{p=0}^{\ell-1}r^{2\ell-2p-1}\Upsilon_{2p+1,J} \bigg(\frac{1}{(2\ell-2p-1)!} - \frac{1}{(2j-2\ell+1)(2\ell-2p-2)!} \bigg).
 \end{align*}
We are then left with showing that the sum
\[
\sum_{\ell=1}^{j} \frac{r^{2j-2\ell}}{(2j-2\ell)!}\sum_{p=0}^{\ell-1}r^{2\ell-2p-1}\Upsilon_{2p+1,J} \bigg(\frac{1}{(2\ell-2p-1)!} - \frac{1}{(2j-2\ell+1)(2\ell-2p-2)!} \bigg)
\]
vanishes. After switching the orders of summation, one gets
 \begin{align*}
  \sum_{p=0}^{j-1} r^{2j-2p-1}\Upsilon_{2p+1,J}  \sum_{\ell=p+1}^{j} \beta_{\ell}(p,j), 
   \end{align*} 
   where
\[
\beta_{\ell}(p,j) = \frac{1}{(2j-2\ell)!(2\ell-2p-1)!} - \frac{1}{(2j-2\ell+1)!(2\ell-2p-2)!}.
\]
It remains to realize that the inner sum over the $\beta_{\ell}$'s is zero. We first rescale it so that
   \begin{align}\label{sommagamma}
&    \sum_{\ell=p+1}^{j}\beta_\ell(p,j)  = \sum_{m=1}^{N} \gamma_{m}, 
   \end{align}
   where  $N=j-p$ and (we do not stress the dependence of $\gamma_{m}$ on $N$)
   \[
\gamma_{m} =   \frac{1}{(2N-2m)!(2m-1)!} - \frac{1}{(2N-2m+1)!(2m-2)!}.
   \]
 If $N=2h+1$ is odd, then the sum in~\eqref{sommagamma} vanishes because $\gamma_{m} = - \gamma_{2h+2-m}$ for all $m=1,\dots, h $ while $\gamma_{h+1}=0$. If $N=2h$ is even, then the sum in~\eqref{sommagamma} vanishes because $\gamma_{m} = - \gamma_{2h+1-m}$ for all $m=1,\dots, h $. 
\end{proof}

We are now ready to prove the key result which will allow us to get the estimates in Proposition~\ref{prop:spaceder} for small radii. Roughly speaking, we ``smoothen'' the singularities coming from the derivatives of $r$ in the Faà di Bruno formula by introducing the functions $\Phi$'s above, and then highlight the cancellations occurring in what is left by means of the functions $\Xi$'s.

\begin{proposition}\label{propodd}
Suppose $k\in\N$  and $J\in \{ 0,\dots, n-1\}^{k}$. Then
\begin{itemize}
\item[i)] $|\Upsilon_{2j+1,J}(r)|\lesssim r^{-k+2j+1}$ for $r\leq 1$ and $j=0, \dots, \big[\frac{k-1}{2}\big]$;
\item[ii)] if $k$ is odd, then
\begin{align}\label{keydecdisp}
\bX_{J} h_{t} &=  \sum_{j=0}^{(k-1)/2} \Phi_{k-2j-1,2j,t} \Upsilon_{2j+1,J} +\sum_{j=1}^{(k-1)/2} h_{t}^{(2j)} \Xi_{2j,J};
\end{align}
\item[iii)] if $k$ is even, then
\begin{align}\label{keydecpari}
\bX_{J} h_{t} &=  \sum_{j=0}^{k/2-1} \Phi_{k-2j-2,2j,t} \Upsilon_{2j+1,J} +\sum_{j=1}^{k/2-1} h_{t}^{(2j)} \Xi_{2j,J} + h_{t}^{(k)} \sigma_{k,J}.
\end{align} 
\end{itemize}
\end{proposition}

\begin{proof}
Since $|\sigma_{j,J}| \lesssim r^{-(k -j)}$ when $r\leq 1$ by~\eqref{estimatesigmajJ}, the estimate $|\Upsilon_{2j+1,J}|\lesssim r^{-k+2j+1}$ follows and (i) is proved.

We first consider the case when $k$ is odd. We assume $k\geq 3$, otherwise the statement is trivial. By Lemma~\ref{Faamisto},
\begin{align*}
\bX_{J} h_{t} 
&= \sum_{j=1}^{k} h_{t}^{(j)}  \sigma_{j,J}  \nonumber \\
&= \sum_{j=1}^{k-1}h_{t}^{(j)}  \sigma_{j,J}   + \Phi_{0,k-1,t} \sigma_{k,J} = R_{1} + \Phi_{0,k-1,t}\Upsilon_{k,J},
\end{align*}
where $R_{1}$ is the above sum over $j=1,\dots, k-1$. We now proceed as follows. We start from $\sigma_{1,J}$, namely the most singular term; we add and subtract the quantity
\[
\sigma_{1,J} \left( \sum_{\ell=1}^{k-2} \frac{(-1)^{\ell} r^{\ell}}{\ell!}h_{t}^{(\ell+1)} \right).
\]
The part that we add, together with the term corresponding to $j=1$ in $R_{1}$, gives $\sigma_{1,J}\Phi_{k-1,0,t}$. Then we distribute each term of the part that we subtract to the term with same order of derivative in $h_{t}$ in the sum over the remaining $j\geq 2$. To be more explicit, we get
\begin{align*}
R_{1} 
&= \sigma_{1,J}\Phi_{k-1,0,t} + h_{t}^{(2)}\big( \sigma_{2,J} + r \sigma_{1,J}\big) + \sum_{j=3}^{k-1} h_{t}^{(j)} \bigg( \sigma_{j,J} - \frac{(-1)^{j-1}}{(j-1)!} r^{j-1} \sigma_{1,J}\bigg)\\
& =  \Upsilon_{1,J}\Phi_{k-1,0,t} + h_{t}^{(2)} \Xi_{2,J} + R_{3},
\end{align*}
where $R_{3}$ is the above sum over $j=3,\dots, k-1$. We proceed recursively. We will now add and subtract the quantity
 \[
\bigg( \sigma_{3,J}   - \frac{1}{2!} r^{2}\Upsilon_{1,J}\bigg) \left( \sum_{\ell=1}^{k-4} \frac{(-1)^{\ell} r^{\ell}}{\ell!}h_{t}^{(\ell+3)} \right)
\]
so to get
\begin{align*}
R_{3} 
&= \Phi_{k-3,2,t}\bigg( \sigma_{3,J}  - \frac{1}{2!} r^{2} \Upsilon_{1,J}\bigg)  + h_{t}^{(4)}\bigg[ \sigma_{4,J}  + \frac{1}{3!} r^{3} \Upsilon_{1,J} + r \bigg( \sigma_{3,J}  - \frac{1}{2!} r^{2} \sigma_{1,J} \bigg)\bigg] \\
& \quad + \sum_{j=5}^{k-1} h_{t}^{(j)} \bigg[ \sigma_{j,J}   - \frac{(-1)^{j-1}}{(j-1)!} r^{j-1} \sigma_{1,J} - \frac{(-1)^{j-3}}{(j-3)!} r^{j-3}\bigg( \sigma_{3,J}  - \frac{1}{2!} r^{2} \sigma_{1,J}\bigg) \bigg]\\
& =  \Phi_{k-3,2,t}\Upsilon_{3,J} + h_{t}^{(4)} \Xi_{4,J} + R_{5},
\end{align*}
where 
\[
 R_{5}=\sum_{j=5}^{k-1} h_{t}^{(j)} \bigg[ \sigma_{j,J} - \frac{(-1)^{j-1}}{(j-1)!} r^{j-1} \Upsilon_{1,J} - \frac{(-1)^{j-3}}{(j-3)!} r^{j-3}\Upsilon_{3,J} \bigg].
\]
If we do this $(k-1)/2$ times, we obtain precisely~\eqref{keydecdisp}.
   \smallskip
   
 We now consider the case when $k$ is even. The proof is similar and we omit some details.   We assume $k\geq 4$, since when $|J|=2$
\[
\bX_{J} h_{t}=h_{t}^{(1)}\sigma_{1,J} + h_{t}^{(2)}\sigma_{2,J} = \Phi_{0,0,t}\sigma_{1,J} + h_{t}^{(2)}\sigma_{2,J}
\]
and there is nothing to prove.
 
By Lemma \ref{Faamisto},
\begin{align*}
\bX_{J} h_{t} 
&= \sum_{j=1}^{k} h_{t}^{(j)}  \sigma_{j,J}  \nonumber \\
&= \sum_{j=1}^{k-2}h_{t}^{(j)}  \sigma_{j,J}   + h_{t}^{(k-1)} \sigma_{k-1,J}+ h_{t}^{(k)}  \sigma_{k,J}  \nonumber \\
&= P_{1} + h_{t}^{(k-1)} \Upsilon_{k-1,J} +  h_{t}^{(k)}  \sigma_{k,J}.
\end{align*}
We now proceed as in the odd case. We first add and subtract the quantity
\[
\sigma_{1,J} \left( \sum_{\ell=1}^{k-3} \frac{(-1)^{\ell} r^{\ell}}{\ell!}h_{t}^{(\ell+1)} \right).
\]
The part that we add, together with the term corresponding to $j=1$ in $R_{1}$, gives $\sigma_{1,J}\Phi_{k-2,0,t}$. Then we distribute each term and proceed recursively as before; at the second step we add and subtract the quantity
 \[
\bigg( \sigma_{3,J}   - \frac{1}{2!} r^{2}\Upsilon_{1,J}\bigg) \left( \sum_{\ell=1}^{k-5} \frac{(-1)^{\ell} r^{\ell}}{\ell!}h_{t}^{(\ell+3)} \right),
\]
and so on; if we do this $k/2-1$ times, we obtain precisely~\eqref{keydecpari}. The proof is then complete.
\end{proof}

\begin{corollary}\label{corodd}
Suppose $k\in\N$ and $J \in \{ 0,\dots, n-1\}^{k}$. Then
\[
|\bX_{J} h_{t}| \lesssim  \tilde \Psi_{k}(r,t) h_{t}(r), \qquad t>0, \, r\leq 1.
\]
\end{corollary}

\begin{proof}
Observe that $| \Xi_{2j,J}| \leq C$ for some $C=C(k)$, all $j=1, \dots, \big[\frac{k-1}{2}\big]$ and $r\leq 1$, as they are smooth by Lemma~\ref{xismooth}. Then one needs to combine Proposition~\ref{propodd} with Lemma~\ref{lemmaphi} and Theorem~\ref{teo:main}~(1).  Suppose, e.g., that $k$ is odd. Then, for $r\leq 1$,
\begin{align}\label{derpsitilde}
|\bX_{J} h_{t}(r)|
&\lesssim \sum_{j=0}^{(k-1)/2} |\Phi_{k-2j-1,2j,t}(r)| |\Upsilon_{2j+1,J}(r)| +\sum_{j=1}^{(k-1)/2} |h_{t}^{(2j)}(r)| \nonumber \\
& \lesssim h_{t}(r) \bigg(\sum_{j=0}^{(k-1)/2} \Psi_{k}(r,t)+\sum_{j=1}^{(k-1)/2} \Psi_{2j}(r,t)\bigg)\\
& \lesssim h_{t}(r) \tilde \Psi_{k}(r,t).\nonumber
\end{align}
The case when $k$ even is similar, and can be proved in the exact same manner.
\end{proof}
Corollary~\ref{corodd} above and Proposition~\ref{prop:dermag1} prove the first part of Proposition~\ref{prop:spaceder}, as we shall explicitly observe below. 

It remains to show the second part, namely that when the derivatives are not mixed one gets a better estimate; namely $\tilde \Psi_{k}$ can be replaced by $\Psi_{k}$. This is unfortunately not possible, in general, if the derivatives are mixed, as we shall see in Remark~\ref{rem:nor} below.

\begin{proposition}\label{prop:extrar}
Suppose $k,\ell \in \mathbb N$, and $j\in \{0,\dots, n-1\}$. Then there exists a bounded positive function $\psi_{k,\ell}$ on $[0,1]$ such that
\[
     |\bX_{j}^{2k+1}({r^{2\ell}})|\lesssim r \, {\psi_{k,\ell}}(r), \qquad r \leq 1.
\]
\end{proposition}
\begin{proof}
We begin by observing that 
\begin{align*}
    \cosh r -1 = r^2\bigg(\frac12 + \sum_{n=2}^\infty \frac{r^{2(n-1)}}{(2n)!}\bigg), \qquad r>0.
\end{align*}
If we write 
\[
g(r)= \frac{1}{\frac{1}{2}+\sum_{n=2}^\infty \frac{r^{n-1}}{(2n)!}},
\] 
then $g \in C^{\infty}([0,\infty))$ and $r^2= (\cosh r-1)g(r^2)$. 
By the Faà di Bruno formula~\eqref{Faa}
\[
\bX_{j}^{2k+1}( r^{2\ell}) =(2k+1)! \sum_{q=1}^{2k+1}\frac{1}{q!}  \bigg(\frac{\partial^{q} }{\partial s^{q}} s^{\ell} \bigg)(r^{2})  \sum_{m_{1}+\cdots + m_{q}= 2k+1}  \frac{\bX_{j}^{m_{1}}(r^{2})}{m_{1}!} \cdots \frac{\bX_{j}^{m_{q}}(r^{2})}{m_{q}!}.
\]
Since for all $q=1,\dots, 2k+1$ at least one $m_{h}$ in the inner sum, $h=1,\dots, q$, is odd, it is enough to prove that the statement holds when $\ell=1$.

If $k=0$  then $\bX_{j}(r^2)=2r \bX_{j} r$, and since $\bX_{j}r$ is bounded by Theorem~\ref{teo:derd} the statement follows. Assume that for some $k \ge 0$ the result holds. Then, 
\begin{equation}\label{XX}
\begin{split}
      \bX_{j}^{2k+3}(r^2)&=  \bX_{j}^{2k+3}((\cosh r-1)g(r^2))\\ 
      &= \sum_{h=0}^{2k+3} \binom{2k+3}{h} \bX_{j}^{2k+3-h}(\cosh r-1)\bX_{j}^{h}(g(r^2)).
      \end{split}
\end{equation}
Now we recall that  for every $m \in \mathbb N$ (see~\eqref{derivative0coshd} and the subsequent discussion)
\begin{equation}\label{dercoshm}
\begin{split}
    &\bX_0^{2m+1} \cosh r= \bX_0\cosh r, \\ 
    &\bX_0^{2(m+1)} \cosh r= \cosh r-|x|^2/4 \\ 
    &\bX_j^{5+m} \cosh r=0,    \qquad j \ne 0, \\ 
    &\bX_j^{3+m} \cosh r=0 \qquad j=\mu+1,...,n-1,\\
    & | \bX_j^3 \cosh r|\lesssim \sqrt{a}|x|  \lesssim r, \qquad j=1,...,\mu.
\end{split}
\end{equation} 
Suppose  $j=0$.  Then, by splitting the cases when $m$ is odd or even in~\eqref{XX} and by~\eqref{dercoshm}
\begin{align*}
     \bX_0^{2k+3}(r^2)&=(\cosh r-1)\bX_0^{2k+3}(g(r^2)) + \sum_{h=0}^{k+1}  \binom{2k+3}{2h}(\bX_0 \cosh r )\bX_0^{2h}(g(r^2))\\ 
     & \qquad +\sum_{h=0}^{k} \binom{2k+3}{2h+1}\bigg(\cosh r-\frac{|x|^2}{4}\bigg)\bX_0^{2h+1}(g(r^2)).
     \end{align*} 
Observe now that, if $r\leq 1$, then $|\bX_0 \cosh r|=|\sinh r| |\bX_0 r| \approx r |\bX_0 r|$ and $| \cosh r -1| \approx r^2$; while for $h=0, \dots, k$, by the Faà di Bruno formula
\begin{align}\label{oddg}
    \bX_0^{2h+1} g(r^2)= (2h+1)! \sum_{j=1}^{2h+1} \frac{1}{j!}g^{(j)}(r^2) \!\!\! \sum_{m_1+\cdots +m_j=2h+1} \frac{\bX^{m_1}(r^2)}{m_{1}!} \cdots \frac{\bX^{m_j}(r^2)}{m_j!},
\end{align} 
and since $2h+1$ is odd, at least one among the $m_q$'s  ($q=1, \dots, j$) is odd. Moreover, $m_q \leq 2k+1$ for all $q$'s, whence the statement follows by the inductive assumption.

Suppose now $j \ne 0$. By~\eqref{XX} and~\eqref{dercoshm}
\begin{align*}
\bX_j^{2k+3}(r^2)=   & (\cosh r-1) \bX_j^{2k+3}(g(r^2)) +  \sum_{h=1}^{4} \binom{2k+3}{h}(\bX_j^{h}\cosh r )\bX_j^{2k+3-h}(g(r^2)).
   \end{align*} 
We conclude by observing that $|\bX_j \cosh r|\approx r |\bX_j r|$ and $|\bX_j^3 \cosh r|\lesssim r$ by~\eqref{dercoshm}, while the terms corresponding to $h=2$ and $h=4$ can be estimated by means of~\eqref{oddg} as above.
\end{proof} 

\begin{proof}[Proof of Proposition~\ref{prop:spaceder}]
To get the first estimate, just combine Corollary~\ref{corodd} above and Proposition~\ref{prop:dermag1}. To get the second estimate, it is enough to observe that by Proposition~\ref{prop:extrar} and~\eqref{eqXi0}, if $J=\{j\}^{k}$ then $ |\Xi_{2\ell,J}(r)| \leq Cr$ for some $C=C(k)$ and all $\ell= 1, \dots , \big[\frac{k-1}{2}\big]$; thus, e.g.\ when $k$ is odd, an extra $r$ appears in the second sum in~\eqref{derpsitilde}, and this gives $\Psi$ in place of $\tilde \Psi$. The case $k$ even is analogous.
\end{proof}

We finally prove that in Proposition~\ref{prop:spaceder}, equivalently in Theorem~\ref{teo:main} (2), the function $\tilde \Psi$ cannot be replaced by $\Psi$ if one considers compositions of an odd amount of different vector fields. 
\begin{remark}\label{rem:nor}
Pick two integers $1\le \ell \le \mu$ and $\mu+1 \le m\le n-1$. Then
\[
\begin{split}
  \bX_\ell^{2} \bX_m  (r^2)
  &= (\bX_\ell^{2} \bX_m  \cosh r)g(r^2)+2 (\bX_\ell \bX_m \cosh r)\bX_\ell g(r^2) +(\bX_m \cosh r)\bX_\ell^{2}g(r^2)\\  & \qquad+(\bX_\ell^{2} \cosh r)\bX_m  g(r^2)+(\bX_\ell \cosh r)\bX_\ell \bX_m  g(r^2)\\ &\qquad  +(\bX_\ell \cosh r)\bX_\ell \bX_m  g(r^2)+(\cosh r-1)\bX_\ell^{2}\bX_m  g(r^2).
\end{split}
\]
Observe that all the terms in the right hand side above tend to zero as $r \to 0^+$, except $(\bX_\ell^{2} \bX_m \cosh r)g(r^2)$ for which the estimate 
\[
 |(\bX_\ell^{2}\bX_m \cosh r)g(r^2)|\approx  a \approx 1, \qquad r \leq 1
 \] 
 holds. Therefore, $| \bX_\ell^{2} \bX_m  (r^2)|\approx 1$ for $r \leq 1$.

Write then $J=(\ell, \ell, m)$ where $\ell$ and $m$ are as above and suppose $r \big(1+\frac{1+r}{t}\big) \leq 1$. In particular, $r \leq 1$ and  $r^{2}\big(1+\frac{1+r}{t}\big) \leq 1$. By Proposition~\ref{propodd} and~\eqref{eqXi}
\[
\bX_{J} h_{t} =  \sum_{j=0}^{1} \Phi_{3-2j-1,2j,t} \Upsilon_{2j+1,J} + \frac{1}{2}h_{t}^{(2)} \bX_{J}(r^{2})
\]
where, by Lemma~\ref{lemmaphi},
\[
\Bigg|\sum_{j=0}^{1} \Phi_{3-2j-1,2j,t} \Upsilon_{2j+1,J}  \Bigg| \lesssim  r \bigg(1+\frac{1+r}{t}\bigg)^{2}  h_{t}(r)
\]
while, since $| \bX_{J}  (r^2)|\approx 1$, if $(1+r)/t $ is large enough then by Proposition~\ref{prop: sharpness}
\[
\bigg|\frac{1}{2}h_{t}^{(2)} \bX_{J}(r^{2})\bigg| \gtrsim  \bigg(1+ \frac{1+r}{t}\bigg)h_{t}(r).
\] 
Therefore, e.g.\ in the regime when $t=\sqrt{r}$ and $r\to 0^{+}$, one has $(1+r)/t \to \infty$ and
\[
|\bX_{J} h_{t}(r) | \gtrsim \bigg(1+ \frac{1+r}{t}\bigg) h_{t}(r) \approx \tilde \Psi_{3}(r,t) h_{t}(r),
\]
whence the estimate $|\bX_{J} h_{t}(r) | \lesssim \Psi_{3}(r,t) h_{t}(r)$ does not hold.
\end{remark}
\subsection{Time derivatives}
We begin with the simple observation that since $\partial_{t}\e^{-t\Ls} = - \Ls \e^{-t\Ls} $ for all $t>0$, one has
\begin{equation}\label{timederL}
\partial_{t}^{k}h_{t} = (-1)^{k} \Ls^{k} h_{t}, \qquad t>0.
\end{equation}
Moreover, $\partial_{t}^{k}\bX h_{t} =\bX \partial_{t}^{k} h_{t}$ for any vector field $\bX$ in $\mathfrak{s}$. By~\eqref{timederL}, Proposition~\ref{prop:spaceder}, and the simple observation that $\tilde \Psi_{\ell} \lesssim \tilde \Psi_{h}$ if $\ell \leq h$,
\begin{align*}
|\partial_{t}^{k}h_{t} (r)|= | \Ls^{k} h_{t}(r)| \leq \sum_{ k \leq |J|\leq 2k}| \bX_J h_{t}(r)| \lesssim \tilde \Psi_{2k}(r,t) h_{t}(r).
\end{align*}
Analogously, we prove Theorem~\ref{teo:main}~(2).

\begin{proof}[Proof of Theorem~\ref{teo:main}~(2)]
Just observe that, for $k,m\in \N$ and $J\in \{0,\dots, n-1\}^{k}$
\[
|\partial_{t}^{m}\bX_{J} h_{t} (r)| = |\bX_{J} \Ls^{m} h_{t} (r)| \leq \sum_{ m+k \leq |I|\leq 2m+k}| \bX_I h_{t}(r)| \lesssim \tilde \Psi_{2k+m}(r,t)h_{t}(r)
\]
by Proposition~\ref{prop:spaceder}. The second  half of the statement is part of Proposition~\ref{prop:spaceder} itself.
\end{proof}

We conclude this section by showing that some of our results can be ``transferred'' to analogous estimates of the heat kernel of the so-called \emph{distinguished Laplacian} 
\[
\Delta =-\sum_{j=0}^{n-1} \bX_j^2,
\]
which is essentially self-adjoint on $L^{2}(\rho)$ and is related to the Laplace--Beltrami operator $\Ls$ by the identity, see e.g.~\cite[Proposition~2]{Ast},
\[
\delta^{-1/2}{\Delta}\,\delta^{1/2}f=(\Ls-{Q^2}/{4}) f
\]
for all (sufficiently regular) radial functions $f$ on $S$. In particular, if $h_{t}^{\Delta}$ denotes the heat kernel of $\Delta$, i.e.\ the convolution kernel of the heat semigroup $\e^{-t\Delta}$ generated by $\Delta$, then
\begin{equation}\label{relazionenuclei}
 h_{t}^{\Delta} = \delta^{1/2} \e^{t \frac{Q^{2}}{4}} h_{t}.
\end{equation}
Since the modular function $\delta$ is not radial, neither is $h_{t}^{\Delta}$. We have the following.

\begin{corollary}
Suppose $m,k\in \N$. There exists $C_{m,k}>0$ such that for all $J\in \{0,\dots, n-1\}^{k}$
\[
\bigg|  \frac{\partial^{m} }{\partial t^{m}}  \bX_{J} h_{t}^{\Delta}(\x)\bigg| \leq C_{m,k} \tilde \Psi_{2m+k}(r,t) \, h_{t}^{\Delta}(\x), \qquad \forall \, \x\in S,\, t>0.
\]
\end{corollary}
\begin{proof}
Observe first that since the modular function $\delta$ is a character of $S$,
\[
\bX_{j} \delta^{1/2} = c_{j} \delta^{1/2}, \qquad c_{j} = (\bX_{j}\delta^{1/2})(e), \qquad j\in \{0,\dots, n-1\}.
\]
Therefore, for all $k\in \N$ and $J\in \{0,\dots, n-1\}^{k}$
\[
\bX_{J} \delta^{1/2} = c_{J} \delta^{1/2}, \qquad c_{J} =  \prod_{j=1}^{k}(\bX_{J_{j}}\delta^{1/2})(e).
\]
Thus, by Proposition~\ref{prop:spaceder} and~\eqref{relazionenuclei}
\begin{align*}
|\partial_{t}^{m}\bX_{J} h_{t}^{\Delta} (\x)|  = | \bX_{J} \Delta^{m} h_{t}^{\Delta} (\x)|
&  \leq \e^{\frac{Q^{2}}{4}t}\sum_{|I| = 2m+k}| \bX_I ( \delta^{1/2}(x) h_{t}(r))|\\
& \lesssim \e^{\frac{Q^{2}}{4}t}  \delta^{1/2} (\x) \sum_{|I| \leq 2m+k}| \bX_{I}h_{t}(r)|\\
& \lesssim \tilde \Psi_{2m+k}(r,t )\, \e^{\frac{Q^{2}}{4}t}  \delta^{1/2} (\x)  h_{t}(r)\\
& =  \tilde \Psi_{2m+k}(r,t)   h_{t}^{\Delta}(\x).
\end{align*}
The proof is complete.
\end{proof}

\section{Asymptotics}\label{sec: 5}
   In this section we shall find the asymptotic behavior of all the radial derivatives of $h_{t}$ when $(1+r)/t$ is large, and in the particular case when $t$ is fixed and $r\to\infty$, also of the time derivatives. In other words, we will prove Theorem~\ref{teo:main}~(3). The proof of the first part is inspired by~\cite{GM}; we shall observe, however, that if one is interested only in the estimates for $r\to \infty$ and fixed $t>0$ then the proof can be slightly simplified, see Remark~\ref{GMsimplified} below.
   
   For notational convenience, we shall denote
  \[
\Rs_{p,q} =   (-\mathcal{R})^q \left(- \frac{1}{\sinh \frac r2}\frac{\partial}{\partial r}\right)^p, \qquad  \eta_{p,q}(r)=\Rs_{p,q} (-r^2).
    \]
 By simple computations, see e.g.~\cite[Eq. (5.23)]{ADY},
\begin{align}\label{asymgauss}
   \Rs_{p,q}  \e^{-\frac{r^2}{4t}}=\sum_{j=1}^{p+q} a_j(r) t^{-j} \e^{-\frac{r^2}{4t}},
\end{align}  
where (recall~\eqref{thetapq})
\begin{align*}
a_{p+q}(r)=4^{-(p+q)}\eta_{1,0}^p(r)\eta_{0,1}^q(r) 
&= 2^{-(p+q)} \bigg(\frac{r}{\sinh \frac{r}{2}} \bigg)^{p} \bigg( \frac{r}{\sinh r}\bigg)^{q}\\
&=(1+r)^{p}\big(\tfrac{1}{2} +r\big)^{q} \e^{-(\frac{p}{2}+q)r} \Theta_{p,q}(r),
\end{align*}
while for $j=1,...,p+q-1$,
\[
|a_j(r)|\lesssim (1+r)^j \e^{-(\frac{p}{2}+q)r}.
\]
Thus
\[
|a_j(r)|t^{-j}\lesssim \bigg(\frac{1+r}{t}\bigg)^j \e^{-(\frac{p}{2}+q)r}, \qquad j=1,...,p+q-1.
\]
By \eqref{asymgauss} we obtain, for $(1+r)/t \geq 1$,
\begin{align}\label{asintoticogauss}
   \Rs_{p,q} \e^{-\frac{r^{2}}{4t}}  =  \e^{-\frac{r^{2}}{4t}} \e^{-(q+\frac{p}{2})r}\bigg(\frac{1+r}{t}\bigg)^{p}\bigg(\frac{\frac{1}{2} +r}{t}\bigg)^{q} \Theta_{p,q}(r) \bigg[ 1+ O\bigg( \frac{t}{1+r}\bigg)\bigg].
\end{align} 
The above estimate~\eqref{asintoticogauss} is almost all we need, together with Lemmas~\ref{lem: radial derivates} and~\ref{lem: heatderapprox}, to prove the first part of Theorem~\ref{teo:main}~(3) when $\nu$ is even. When $\nu$ is odd, the following lemma will be particularly useful; see also~\cite[Lemma 5]{GM0}. Its proof is elementary and omitted. 

\begin{lemma}\label{GMinsp}
For $p,q\in \N$ define the function
\begin{align*}
E_{p,q}(r,s) &= (\cosh s -\cosh r)^{-\frac{1}{2}} \bigg(\frac{s}{\sinh s} \bigg)^{q} \bigg(\frac{s}{\sinh \frac{s}{2}} \bigg)^{p}\\ &\qquad \qquad \qquad - \sqrt{2} (s^{2} - r^{2})^{-\frac{1}{2}} \bigg(\frac{r}{\sinh r} \bigg)^{q+\frac{1}{2}} \bigg(\frac{r}{\sinh \frac{r}{2}} \bigg)^{p},
\end{align*}
and suppose $r>0$ and $s\in (r,r+2)$. Then
\begin{itemize}
\item[(i)] $|E_{p,q}(r,s) | \lesssim (s^{2}-r^{2})^{\frac{1}{2}} (1+r)^{p+q-\frac{1}{2}} \e^{-r(\frac{p}{2} + q + \frac{1}{2})}$;
\item[(ii)] $| \partial_{s}E_{p,q}(r,s) | \lesssim s (s^{2}-r^{2})^{-\frac{1}{2}} (1+r)^{p+q-\frac{1}{2}} \e^{-r(\frac{p}{2} + q + \frac{1}{2})}$.
\end{itemize}
\end{lemma}
We are now ready to prove Theorem~\ref{teo:main}~(3).
    
\begin{proof}[Proof of Theorem~\ref{teo:main}~(3)]
By Lemmas~\ref{lem: radial derivates} and~\ref{lem: heatderapprox}
\[
    \frac{\partial^k}{\partial r^k} h_t(r) = \sinh^k r \Rs^k h_t(r) + O\bigg[ h_{t}(r) \bigg(\frac{1+r}{t}\bigg)^{k-1}\bigg].
\]
We claim that
    \begin{equation}\label{claimmain}
    \begin{split}
 \sinh^k r \Rs^k h_t(r) &=  (-1)^{k} \mathfrak{c}_{0} t^{-\frac12} \e^{-\frac{Q^2}{4}t -\frac{r^2}{4t}-\frac{Q}{2} r }  (\e^{-r}\sinh r)^{k}\bigg(\frac{1+r}{t}\bigg)^{\frac{\mu}{2}} \\ & \quad \qquad\times \bigg(\frac{\frac{1}{2} +r}{t}\bigg)^{k + \frac{\nu}{2}} \Theta_{\frac{\mu}{2},k+\frac{\nu}{2}}(r) \bigg[ 1+ O\bigg( \sqrt{\frac{t}{1+r}}\bigg)\bigg].
    \end{split}
\end{equation} 
Once the claim is proved, by~\eqref{boundsht}
\[
h_t(r) \bigg(\frac{1+r}{t}\bigg)^{k-1}= (-1)^{k} \sinh^k r \Rs^k h_t(r) O\bigg( \frac{t}{1+r}\bigg)
\]
and the first part of the theorem follows. Therefore, we shall prove~\eqref{claimmain}.
 
Suppose that $\nu$ is even. By~\eqref{asintoticogauss}
\begin{align*}
   \Rs^{k} h_t(r) 
   &= (-1)^{k} \mathfrak{c}_{0} t^{-\frac12} \e^{-\frac{Q^2}{4}t} \Rs_{\frac{\mu}{2},\frac{\nu}{2}+k}  \e^{-\frac{r^2}{4t}}\\
&= (-1)^{k} \mathfrak{c}_{0} t^{-\frac12} \e^{-\frac{Q^2}{4}t-\frac{r^2}{4t}-\frac{Q}{2} r - kr} \\
& \qquad \qquad \times \bigg(\frac{1+r}{t}\bigg)^{\frac{\mu}{2}}\bigg(\frac{\frac{1}{2} +r}{t}\bigg)^{k + \frac{\nu}{2}} \Theta_{\frac{\mu}{2},k+\frac{\nu}{2}}(r) \bigg[ 1+ O\bigg( \frac{t}{1+r}\bigg)\bigg],
\end{align*}
and the claim~\eqref{claimmain} follows.

Suppose now that $\nu$ is odd. In this case, it is notationally helpful to get rid of the function $\Theta$. By~\eqref{Rkht}, the claim follows if we prove
\begin{equation}\label{claimmain2}
 f_{k}(r) =  (2t)^{- k - \frac{\nu}{2} - \frac{\mu}{2}} \sqrt{\pi}  \bigg(\frac{r}{\sinh r} \bigg)^{k+\frac{\nu}{2}} \bigg(\frac{r}{\sinh \frac{r}{2}} \bigg)^{\frac{\mu}{2}} \e^{-\frac{r^{2}}{4t}} \bigg(1+O\bigg(\sqrt{\frac{t}{1+r}} \bigg)\bigg).
\end{equation}
By~\eqref{asintoticogauss}
  \begin{align*}
    f_{k}(r) &=(2t)^{- k - \frac{\nu+1}{2} - \frac{\mu}{2}} I_{k}(r)\bigg[ 1+O\bigg(\frac{t}{1+r}\bigg)\bigg] ,
  \end{align*}   
where 
\[
I_{k}(r) = \int_{r}^{\infty } \frac{1}{\sqrt{\cosh s-\cosh r}} \bigg(\frac{s}{\sinh \frac{s}{2}}\bigg)^{\frac{\mu}{2}}\bigg(\frac{s}{\sinh s}\bigg)^{k+\frac{\nu+1}{2}-1}s\, \e^{-\frac{s^{2}}{4t}}  \, ds.
\]
We split $I_{k}(r)$ into the sum of the integral $I_{k,1}(r)$ on $(r,r+2)$ and $I_{k,2}(r)$ on $(r+2,\infty)$. By Lemma~\ref{GMinsp}, whose notation we maintain throughout the proof,
 \begin{align*}
I_{k,1}(r)
 & =  \sqrt{2}  \bigg(\frac{r}{\sinh r} \bigg)^{k+\frac{\nu}{2}} \bigg(\frac{r}{\sinh \frac{r}{2}} \bigg)^{\frac{\mu}{2}}  \int_{r}^{r+2}(s^{2} - r^{2})^{-\frac{1}{2}}  s\, \e^{-\frac{s^{2}}{4t}}\, ds \\
 & \qquad \qquad + \int_{r}^{r+2} E_{\frac{\mu}{2}, k+\frac{\nu-1}{2}}(s,r) s\, \e^{-\frac{s^{2}}{4t}}\, ds.
 \end{align*}
With the change of variables $u = \frac{s^{2}-r^{2}}{4t}$,
 \begin{align*}
 \int_{r}^{r+2}(s^{2} - r^{2})^{-\frac{1}{2}}  s\, \e^{-\frac{s^{2}}{4t}}\, ds 
 &=\sqrt{t} \, \e^{-\frac{r^{2}}{4t}} \int_{0}^{\frac{r+1}{t}} \frac{1}{\sqrt{u}}  \e^{-u}\, du \\
 & = \sqrt{\pi t} \, \e^{-\frac{r^{2}}{4t}} \bigg(1+O\bigg(\sqrt{\frac{t}{1+r}} \bigg)\bigg),
 \end{align*}
 and as we shall see, this gives the main contribution to $I_{k}(r)$. Integrating by parts the second integral in $I_{k,1}$, as $s\, \e^{-\frac{s^{2}}{4t}} = -2t \frac{\partial }{\partial s} \e^{-\frac{s^{2}}{4t}}$,
 \begin{align*}
 \int_{r}^{r+2} &E_{\frac{\mu}{2}, k+\frac{\nu-1}{2}}(s,r) s\, \e^{-\frac{s^{2}}{4t}}\, ds  \\
 & = -2t\bigg(E_{\frac{\mu}{2}, k+\frac{\nu-1}{2}}(r+2,r) \e^{-\frac{(r+2)^{2}}{4t}} - \int_{r}^{r+2} \frac{\partial }{\partial s}  E_{\frac{\mu}{2}, k+\frac{\nu-1}{2}}(s,r) \e^{-\frac{s^{2}}{4t}}\, ds \bigg).
 \end{align*}
By Lemma~\ref{GMinsp}
  \begin{align*}
 \left|2t E_{\frac{\mu}{2}, k+\frac{\nu-1}{2}}(r+2,r) \e^{-\frac{(r+2)^{2}}{4t}}\right|  \lesssim \sqrt{t} \bigg(\frac{r}{\sinh r} \bigg)^{k+\frac{\nu}{2}} \bigg(\frac{r}{\sinh \frac{r}{2}} \bigg)^{\frac{\mu}{2}} \e^{-\frac{r^{2}}{4t}} \sqrt{\frac{t}{1+r}},
  \end{align*} 
  while
 \begin{align*}
 \bigg| 2t \int_{r}^{r+2} \frac{\partial }{\partial s}  & E_{\frac{\mu}{2}, k+\frac{\nu-1}{2}}(s,r) \e^{-\frac{s^{2}}{4t}}\, ds\bigg|\\
 & \lesssim t \e^{-\frac{r^{2}}{4t}}  (1+r)^{\frac{\mu}{2}+k+\frac{\nu}{2}-1} \e^{-r(\frac{\mu}{4} + k+\frac{\nu}{2})}  \int_{r}^{r+2} s (s^{2}-r^{2})^{-\frac{1}{2}}\, ds
 \\& \lesssim \sqrt{t} \bigg(\frac{r}{\sinh r} \bigg)^{k+\frac{\nu}{2}} \bigg(\frac{r}{\sinh \frac{r}{2}} \bigg)^{\frac{\mu}{2}} \e^{-\frac{r^{2}}{4t}} \sqrt{\frac{t}{1+r}} .
 \end{align*}
Summing up
\[
 I_{k,1}(r) =  \sqrt{2 \pi t}  \bigg(\frac{r}{\sinh r} \bigg)^{k+\frac{\nu}{2}} \bigg(\frac{r}{\sinh \frac{r}{2}} \bigg)^{\frac{\mu}{2}} \e^{-\frac{r^{2}}{4t}} \bigg(1+O\bigg(\sqrt{\frac{t}{1+r}} \bigg)\bigg).
\]
We now estimate $I_{k,2}(r)$. We integrate by parts as before, and we get that $I_{k,2}(r)$ is the sum of the boundary term, which we call $I_{k,3}(r)$, and an integral which we denote by $I_{k,4}(r)$. We get
\begin{align*}
|I_{k,3}(r) |& \lesssim t \frac{1}{\sqrt{\cosh (r+2)-\cosh r}} \bigg(\frac{r+2}{\sinh \frac{r+2}{2}}\bigg)^{\frac{\mu}{2}}\bigg(\frac{r+2}{\sinh (r+2)}\bigg)^{k+\frac{\nu+1}{2}-1} \e^{-\frac{(r+2)^{2}}{4t}}\\
& \lesssim \sqrt{t} \bigg(\frac{r}{\sinh \frac{r}{2}}\bigg)^{\frac{\mu}{2}}\bigg(\frac{r}{\sinh r}\bigg)^{k+\frac{\nu}{2}}  \e^{-\frac{r^{2}}{4t}} \sqrt{\frac{t}{1+r}},
\end{align*}
as well as
 \begin{align*}
|I_{k,4}(r) | &\lesssim  t \bigg| \int_{r+2}^{\infty }  \partial_{s} \bigg( \frac{1}{\sqrt{\cosh s-\cosh r}} \bigg(\frac{s}{\sinh \frac{s}{2}}\bigg)^{\frac{\mu}{2}}\bigg(\frac{s}{\sinh s}\bigg)^{k+\frac{\nu+1}{2}-1}\bigg) \e^{-\frac{s^{2}}{4t}}  \, ds\bigg| \\
  & \lesssim  t\int_{r+2}^{\infty } \frac{1}{\sqrt{\sinh s}}  \bigg(\frac{s}{\sinh \frac{s}{2}}\bigg)^{\frac{\mu}{2}}\bigg(\frac{s}{\sinh s}\bigg)^{k+\frac{\nu-1}{2}} \e^{-\frac{s^{2}}{4t}}\, ds\\
  & \lesssim \sqrt{t} \bigg(\frac{r}{\sinh r} \bigg)^{k+\frac{\nu}{2}} \bigg(\frac{r}{\sinh \frac{r}{2}} \bigg)^{\frac{\mu}{2}} \e^{-\frac{r^{2}}{4t}}  \sqrt{\frac{t}{1+r}}.
 \end{align*}
The claim~\eqref{claimmain2} follows, and the proof of the first part of the statement is complete.

Let us now restrict to the case $r\to \infty$ and $t>0$ fixed. An immediate consequence of what we just proved is that the second part of the statement holds for $m=0$. Thus, we are left with considering the case when $m>0$.

 Recall~\cite{ADY} that the radial part of $\Ls$ is
\begin{align*}
  \mathrm{rad}(\Ls)= - \frac{\partial^2}{\partial r^2} - \left\{\frac{\mu+\nu}{2}\coth{\frac r2}+ \frac{\nu}{2} \tanh{\frac r2}\right\} \frac{\partial}{\partial r}.
\end{align*}
We claim that for every $\ell \in \mathbb N$
\begin{align*}
    \bigg|\frac{\partial^\ell}{\partial r^\ell} \coth{\frac r2}\bigg|, \bigg|\frac{\partial^\ell}{\partial r^\ell} \tanh{\frac r2}\bigg| \lesssim 1, \qquad r>1.
\end{align*} 
Indeed, this easily follows by the identities
\begin{align*}
    \tanh s= \frac{\e^{s}}{\e^{s}+\e^{-s}}-\frac{\e^{-s}}{\e^{s}+\e^{-s}}, \qquad  \coth {s}=\frac{\e^{s}}{\e^{s}-\e^{-s}}+\frac{\e^{-s}}{\e^{s}-\e^{-s}}, \qquad s>0,
\end{align*}
and, for $r>0$,
\begin{align*}
     \frac{\e^{r/2}}{\e^{r/2}+\e^{-r/2}}&= \sum_{j=0}^\infty (-1)^j\e^{-rj}, \qquad &&\frac{\e^{-r/2}}{\e^{r/2}+\e^{-r/2}}=\sum_{j=0}^\infty (-1)^j\e^{-r(j+1)},\\ 
     \frac{\e^{r/2}}{\e^{r/2}-\e^{-r/2}}&=\sum_{j=0}^\infty \e^{-rj}, \  &&\frac{\e^{-r/2}}{\e^{r/2}-\e^{-r/2}}=\sum_{j=0}^\infty \e^{-r(j+1)}.
\end{align*}
This implies that  for $m \in \mathbb N$
\[
    (-1)^m\mathrm{rad}(\Ls)^m= \frac{\partial^{2m}}{\partial r^{2m}}+\sum_{j=1}^{2m-1} c_{j,m}(r) \frac{\partial^j}{\partial r^j},
\]
where $c_{j,m}$ are suitable bounded functions. Therefore
\begin{align*}
  \frac{\partial^{m} }{\partial t^{m}}   \frac{\partial^{k} }{\partial r^{k}} h_{t} (r) 
  &= (-1)^{m} \Ls^{m} \frac{\partial^{k} }{\partial r^{k}} h_{t} (r) \\
  & = (-1)^m \mathrm{rad}(\Ls)^m \frac{\partial^{k} }{\partial r^{k}}  h_t(r) =\bigg[\frac{\partial^{2m+k}}{\partial r^{2m+k}}+\sum_{j=1}^{2m+k-1} c_{j,m}(r) \frac{\partial^j}{\partial r^j}\bigg] h_t(r).
\end{align*} Hence, the asymptotic expansion follows directly by the case $m=0$.
    \end{proof}

\begin{remark}\label{GMsimplified}
When $r\to \infty$ and $t>0$ is fixed, the proof of the asymptotic expansion for the radial derivatives of $h_{t}$ can be proved in a slightly simpler way when $\nu$ is odd. In this case, indeed, one can prove that for $p,q\in \N$
\begin{align}\label{claimasymp}
 \int_{r}^{\infty} \frac{\sinh s}{\sqrt{\cosh s-\cosh r}} \Rs_{p,q} \e^{-\frac{s^{2}}{4t}}  \dd s \sim   \sqrt{\pi}  \e^{-\frac{r^{2}}{4t}} \e^{-(q-\frac{1}{2}+\frac{p}{2})r }  \bigg(\frac{r}{t} \bigg)^{p+q-\frac{1}{2}} .
     \end{align}
 To show this, let $I(r,t)$ be the above integral. By~\eqref{asintoticogauss}, if $r\to \infty$,
\begin{align*}
     I(r,t)  =  t^{-p-q} \e^{-\frac{r^{2}}{4t}} \e^{-(q+\frac{p}{2})r } I_{1}(r,t)\bigg[1 + O\bigg(\frac{t}{r}\bigg)\bigg],
\end{align*}
where
\begin{align*}
I_{1}(r,t) = \int_{0}^{\infty} \frac{\sinh (s+r)}{\sqrt{\cosh (s+r)-\cosh r}} \e^{-\frac{s^{2}}{4t}} \e^{-\frac{rs}{2t}} \e^{-(q+\frac{p}{2})s}(s+r)^{p+q} \, ds. \end{align*}
Observe now that
\begin{align*}
 \frac{\sinh (s+r)}{\sqrt{\cosh (s+r)-\cosh r}} = 2^{-1/2} \e^{s+\frac{r}{2}}\frac{1}{\sqrt{\e^{s}-1}} \big(1+O(\e^{-2r})\big),
\end{align*}
the remainder being uniform in $s>0$, so that
\[
I_{1}(r,t) = 2^{-1/2} \e^{\frac{1}{2}r} I_{2}(r,t) \big(1+O(\e^{-2r})\big),
\]
where
\[
I_{2}(r,t) = \int_{0}^{\infty}  \frac{1}{\sqrt{\e^{s}-1}}  \, \e^{-\frac{s^{2}}{4t}} \e^{-\frac{rs}{2t}} \e^{-(q+\frac{p}{2}-1)s}(s+r)^{p+q} \, ds.
\]
By means of the binomial expansion,
\begin{align*}
I_{2}(r,t) 
&= \sum_{j=0}^{p+q} \binom{p+q}{j} r^{p+q-j} \int_{0}^{\infty}   \frac{1}{\sqrt{\e^{s}-1}}  \e^{-\frac{s^{2}}{4t}} \e^{-\frac{rs}{2t}} \e^{-(q+\frac{p}{2}-1)s}s^{j} \, ds,
\end{align*}
and by Laplace's method~\cite{Erd}, for $j=0, \dots, p+q$ one gets 
    \[
  \int_{0}^{\infty}   \frac{1}{\sqrt{\e^{s}-1}}  \e^{-\frac{s^{2}}{4t}} \e^{-\frac{rs}{2t}} \e^{-(q+\frac{p}{2}-1)s}s^{j} \, ds \sim \Gamma\bigg(j+ \frac{1}{2}\bigg) \bigg(\frac{2t}{r} \bigg)^{j+\frac{1}{2}}, \qquad r\to \infty.
    \]
Since the choice $j=0$ provides the leading term,~\eqref{claimasymp} follows.
 \end{remark}

\section{Maximal and Ornstein--Uhlenbeck operators}\label{sec: 7}
In this section we provide two applications of the estimates given in Theorem~\ref{teo:main}. They concern the weak type $(1,1)$ of certain maximal operators associated with $h_{t}$ and its time derivatives, and the discreteness of the spectrum of certain Riemannian Ornstein--Uhlenbeck operators on $S$.
\subsection{Weak type $(1,1)$ of maximal operators}
In this section we study the weak type $(1,1)$ boundedness of the maximal operators 
\[
\mathcal{H}_j f =\sup_{t>0}\bigg|t^{j}\, \frac{\partial^{j}}{\partial {t}^{j}} \e^{-t\Ls}f\bigg| , \quad  \widetilde{\mathcal{H}}_j f= \sup_{t>0}\bigg|\min(1,t)^{j} \frac{\partial^{j}}{\partial {t}^{j}} \e^{-t\Ls}f\bigg|, \qquad j\in \N, \; f \in L^1.
\]
The operators $\mathcal{H}_{j}$ have a longstanding history which dates back at least to Stein~\cite{SteinLPS}, and their $L^{p}$-boundedness, $1<p<\infty$, has been already studied in depth (see, e.g.,~\cite{L, FP}). As for endpoint results, on real hyperbolic spaces it is known that $\mathcal{H}_1$ is of weak type $(1,1)$, while $\mathcal H_{j}$ is not if $j\geq 2$; cf.~\cite[Theorem 1]{LS}. By means of Theorem~\ref{teo:main}~(2), we shall extend this boundedness result for $\mathcal{H}_{1}$ to Damek--Ricci spaces, and prove that instead $\widetilde{\mathcal{H}}_j$ is of weak type $(1,1)$ for all $j$'s. This motivates and justifies their introduction. Observe indeed that $\widetilde{\mathcal{H}}_j f \leq \mathcal{H}_j f $ for all $j\geq 1$.
\begin{theorem}\label{theorem:maximal}
The operators $ \mathcal{H}_1$ and $\widetilde{\mathcal{H}}_j$ ($j\geq 1$) are of weak type $(1,1)$.
\end{theorem} 
Before we prove this result, we need a more precise estimate of the first time derivative of $h_t$. This is a generalization of~\cite[(5.16)]{LS}, from which we borrow part of the proof, to Damek--Ricci spaces.
\begin{proposition}\label{prop:LS}
For every $r,t>0$, the following holds 
\begin{align*}
    \bigg|\frac{\partial}{\partial t} h_t(r)\bigg| \lesssim \left(\left|\frac{r^2}{4t^2}-\frac{Q^2}{4}\right|+\frac{1}{t}\right)h_t(r).
\end{align*}
\end{proposition} 
\begin{proof}We maintain the notation $\Rs_{p,q}$ introduced in the previous section, and we first assume that $\nu$ is even. In this case,
\begin{align*}
    h_t(r)=\mathfrak{c}_0 t^{-\frac{1}{2}} \e^{-\frac{Q^2}{4} t} \mathcal{R}_{\frac{\mu}{2},\frac{\nu}{2}}\, \e^{-\frac{r^2}{4t}}.
\end{align*} 
Thus
\begin{align*}
  \mathfrak{c}_0^{-1}  \frac{\partial}{\partial t}  h_t(r)&= -\frac{1}{2}t^{-\frac{3}{2}}\e^{-\frac{Q^2}{4} t} \mathcal{R}_{\frac{\mu}{2},\frac{\nu}{2}} \e^{-\frac{r^2}{4t}}- \frac{Q^2}{4}  t^{-\frac12}\e^{-\frac{Q^2}{4} t} \mathcal{R}_{\frac{\mu}{2},\frac{\nu}{2}}\e^{-\frac{r^2}{4t}} \\ 
    & \qquad +  t^{-\frac12}\e^{-\frac{Q^2}{4} t}  \frac{\partial}{\partial t}\mathcal{R}_{\frac{\mu}{2},\frac{\nu}{2}}\e^{-\frac{r^2}{4t}},
\end{align*} 
where by~\eqref{asymgauss}
\[
    \frac{\partial}{\partial t} \mathcal{R}_{\frac{\mu}{2},\frac{\nu}{2}} \e^{-\frac{r^2}{4t}}=-\sum_{j=1}^{(\mu+\nu)/2}ja_j(r)t^{-j-1}\e^{-\frac{r^2}{4t}}+\frac{r^2}{4t^2} \sum_{j=1}^{(\mu+\nu)/2}a_j(r)t^{-j}\e^{-\frac{r^2}{4t}}.
\]
Therefore
\begin{align*}
     &\frac{\partial}{\partial t}  h_t(r) = -\frac{1}{2t}h_t(r)+\bigg[\frac{r^2}{4t^2}-\frac{Q^2}{4}\bigg] h_t(r)- \frac1t \mathfrak{c}_0 t^{-\frac12}\e^{-\frac{Q^2}{4} t}  \sum_{j=1}^{(\mu+\nu)/2}ja_j(r)t^{-j}\e^{-\frac{r^2}{4t}},
\end{align*} 
and in particular,  since $a_{j}$'s are non-negative functions (cf.~\cite{ADY}),
\begin{align*}
    \left| \frac{\partial}{\partial t}h_t(r)\right| \lesssim \bigg(\bigg|\frac{r^2}{4t^2}-\frac{Q^2}{4}\bigg|+\frac{1}{t}\bigg) h_t(r).
\end{align*} 
Next, we assume that $\nu$ is odd. Then
\begin{align*}
     h_t(r)=\mathfrak{c}_0 \pi^{-\frac12} t^{-\frac12}\e^{-\frac{Q^2}{4} t}  \int_{r}^\infty \frac{\sinh s}{\sqrt{\cosh s-\cosh r}} \mathcal{R}_{\frac{\mu}{2},\frac{\nu+1}{2}} \e^{-\frac{s^2}{4t}} \ ds,
\end{align*} 
and by similar computations to those above
\begin{align*}
   \mathfrak{c}_0^{-1} \pi^{\frac12}  \frac{\partial }{\partial t}  h_t(r) &= \bigg[-\frac{1}{2t}+\frac{r^2}{4t^{2}}-\frac{Q^2}{4}\bigg]   \mathfrak{c}_0^{-1} \pi^{\frac12}  h_t(r)\\ &+ t^{-\frac12}\e^{-\frac{Q^2}{4} t}  \int_{r}^\infty \frac{\sinh s}{\sqrt{\cosh s-\cosh r}} \frac{s^2-r^2}{4t^2} \mathcal{R}_{\frac{\mu}{2},\frac{\nu+1}{2}}  \e^{-\frac{s^2}{4t}}\ ds \\ 
    &-  t^{-\frac12}\e^{-\frac{Q^2}{4} t}  \int_{r}^\infty \frac{\sinh s}{\sqrt{\cosh s-\cosh r}} \sum_{j=1}^{n/2}ja_j(s)t^{-j-1}\e^{-\frac{s^2}{4t}} \ ds.
    \end{align*} 
Therefore
\begin{align*}
     \left| \frac{\partial}{\partial t}h_t(r)\right| &\lesssim \bigg|\frac{r^2}{4t^{2}}-\frac{Q^2}{4} \bigg|h_t(r)+ \frac{1}{t}h_t(r) + t^{-\frac12}\e^{-\frac{Q^2}{4} t}  I(t,r),
     \end{align*}
     where,  by \cite[Proposition 5.22]{ADY},
     \begin{align*}
 & I(t,r) =   \int_{r}^\infty \frac{\sinh s}{\sqrt{\cosh s-\cosh r}} \frac{s^2-r^2}{4t^2} \mathcal{R}_{\frac{\mu}{2},\frac{\nu+1}{2}}  \e^{-\frac{s^2}{4t}} \ ds \\
  &\quad  \approx  \int_{r}^\infty \frac{\sinh s}{\sqrt{\cosh s-\cosh r}} \frac{s^2-r^2}{4t^2} \frac{1+s}{t}\bigg(1+\frac{1+s}{t}\bigg)^{\frac{\mu}{2}-1+\frac{\nu+1}{2}}\e^{-(\frac{\mu}{4}+\frac{\nu+1}{2})s-\frac{s^2}{4t}} \ ds.
\end{align*}
One can now follow \cite[Lemma 6]{LS} to obtain that 
\begin{align*}
  t^{-\frac12}\e^{-\frac{Q^2}{4} t}  I(t,r)  \lesssim \frac{1}{t}h_t(r),
\end{align*} which implies the desired conclusion. 
We omit the details.
\end{proof}
\begin{proof}[Proof of Theorem \ref{theorem:maximal}]
To begin with, observe that $\mathcal{H}_j \le \mathcal{H}_j^0+\mathcal{H}_j^\infty$,  where 
\[
\mathcal{H}_{j}^0 f = \sup_{0<t\leq 1}\bigg|f\ast t^j\frac{\partial^j}{\partial t^{j} } h_t\bigg|, \qquad  \mathcal{H}_j^\infty f=\sup_{t > 1}\bigg|f\ast t^j\frac{\partial^j}{\partial t^{j} } h_t\bigg|.
\]
We split $\widetilde{\mathcal{H}}_{j}$ analogously. By general well-known estimates for $h_{t}$, cf.~\cite{VSCC} or~\cite[Lemma 3.2]{BPV}, for all $j\in \N$ there is $c=c(j)>0$ such that
\begin{align}\label{truccolie}
\bigg|t^{j} \frac{\partial^{j}}{\partial {t}^{j}} \e^{-t\Ls} f\bigg|\lesssim  \e^{-ct\Ls} |f|, \qquad t\in (0,1].
\end{align}
Therefore, the weak type $(1,1)$ of $\mathcal{H}_{j}^0 $ follows by that of $\mathcal{H}_{0}^0 $~\cite[Theorem 5.50]{ADY}. Thus, we focus on the boundedness of $\mathcal{H}_1^\infty$ and $\widetilde{\mathcal{H}}_j^\infty$. Observe first that
\[
\sup_{t>1} \bigg|f\ast t \frac{\partial}{\partial t} h_t\bigg| \leq |f|\ast \sup_{t>1} \bigg|t \frac{\partial}{\partial t}h_t\bigg|, \qquad \sup_{t>1} \bigg|f\ast \frac{\partial^{j}}{\partial t^{j}} h_t\bigg| \leq |f|\ast \sup_{t>1} \bigg|\frac{\partial^{j}}{\partial t^{j}} h_t\bigg|.
\]
Suppose $t>1$.  By Proposition \ref{prop:LS}, 
\[
    \bigg|t \frac{\partial}{\partial t} h_t(r)\bigg| \lesssim h_t(r)\left(\left|\frac{r^2}{4t}-\frac{Q^2}{4}t\right|+1\right).
\]
One can then follow closely~\cite[Lemma 8]{LS} to deduce that, for $r \ge 1$,
\begin{align*}
    \sup_{t>1}  \bigg|t \frac{\partial}{\partial t} h_t(r)\bigg|  \lesssim \e^{-Qr},
\end{align*}
and by \cite[Theorem 3.14]{ADY}, this implies that  the convolution operator whose kernel is $\textbf{1}_{B_1^c(e)}\big|t \frac{\partial}{\partial t}  h_t\big|$ is of weak type $(1,1)$. If $r\le 1$, instead,
\begin{align*}
     &\sup_{t>1}\bigg|t \frac{\partial}{\partial t} h_t(r)\bigg| \lesssim  \sup_{t >1} \bigg(\bigg|\frac{r^2}{4t}-\frac{Q^2}{4}t\bigg|+1\bigg)t^{-\frac{3}{2}}r \bigg(1+\frac{r}{t}\bigg)^{\frac{n-3}{2}} \e^{-\frac{r^2}{4t}-\frac{Q}{2}r-\frac{Q^2}{4}t} \lesssim 1,
\end{align*} 
thus $\mathcal{H}_1$ is of weak type (1,1). We now focus on $\widetilde{\mathcal{H}}_j$ for $j \ge 1.$ By~\eqref{truccolie}, it suffices to prove that the operator
\begin{align*}
  f \mapsto f \ast  \sup_{t>1}\bigg| \frac{\partial^j}{\partial t^j} h_t\bigg|
\end{align*}
is of weak type (1,1). By Theorem~\ref{teo:main} $(2)$, for $t>1$
\begin{align*}
    \bigg| \frac{\partial^j}{\partial t^j} h_t(r)\bigg| \lesssim \bigg[1+\bigg(\frac{r}{t}\bigg)^{2j}\bigg]h_t(r), \qquad r>0,
\end{align*}
so it is in turn enough to prove that
\[
    f \mapsto f \ast \bigg|\sup_{t>1}\bigg(\frac{r}{t}\bigg)^{2j}h_t(r)\bigg|
\]
is of weak type (1,1).  If $\frac{r}{t}\le 3Q$, then
   \begin{align*}
     \bigg(\frac{r}{t}\bigg)^{2j}h_t(r)\lesssim h_t(r).
   \end{align*} 
Suppose now $\frac{r}{t}>3Q$. Then
   \begin{align*}
        \bigg(\frac{r}{t}\bigg)^{2j}h_t(r) &\approx   \bigg(\frac{r}{t}\bigg)^{2j}  t^{-\frac32} (1+r)  \left(1+ \frac{1+r}{t}\right)^{\frac{n-3}{2}}  \e^{-\frac{r^2}{4t}-\frac{Q}{2}r-\frac{Q^2}{4}t} \\ 
        &\lesssim  \bigg(\frac{r}{t}\bigg)^{2j+1+\frac{n-3}{2}} \e^{-\frac{r^2}{4t}-\frac{Q}{2}r-\frac{Q^2}{4}t}.
   \end{align*}
Now we observe that
\begin{align*}
    \frac{r^2}{4t}+\frac{Q^2}{4}t+\frac{Q}{2}r \ge r \bigg( \frac{r}{4t}+\frac{Q}{2}\bigg) \ge r  \frac{5Q}{4},
\end{align*} 
whence
\begin{align*}
    \bigg(\frac{r}{t}\bigg)^{2j+1+\frac{n-3}{2}} \e^{-\frac{r^2}{4t}-\frac{Q}{2}r-\frac{Q^2}{4}t} \lesssim r^{2j+\frac{1}{2}+\frac{n-3}{2}} \e^{-\frac{Q5}{4}r}\lesssim \e^{-Qr}.
\end{align*} 
Summing up, we have proved that
\begin{align*}
  \sup_{t> 1}  \bigg|\frac{\partial^j}{\partial t^j} h_t(r) \bigg| \lesssim \sup_{t> 1} h_t(r) + \e^{-Qr}.
\end{align*} An application of \cite[Theorem 3.14, (5.55)]{ADY} gives the desired conclusion.
 \end{proof}

    \subsection{Riemannian Ornstein--Uhlenbeck operators}
Suppose $t>0$, and consider the operator
	\begin{equation*}
		\Ls_{t}= \mathcal{L} - \frac{\nabla {h_t}}{{h_{t}}} \cdot \nabla, \qquad \Dom(\Ls)=C_{c}^{\infty},
	\end{equation*}
	which arises from the Hermitian  form $(\phi,\psi) \mapsto \int_{S} \nabla \phi\cdot \nabla \bar{\psi}\,  h_{t}\, \dd \lambda$, see e.g.~\cite{BC3}. The operator $\Ls_{t}$ can be considered as a Riemannian version on $S$ of the classical Ornstein--Uhlenbeck operator, cf.~\cite{Baudoinetal, LustPiquard, BC1, BC2} and references therein. It is nowadays classical, see e.g.~\cite[Theorem 2.4]{Strichartz}, that $\mathcal{L}_{t}$ is essentially self-adjoint on $L^2(\lambda_{t})$, where $\lambda_{t}$ is the absolutely continuous measure with density $h_{t}$ with respect to $\lambda$. We denote by $\bar{\mathcal{L}}_{t}$ its closure, which is its unique self-adjoint extension. 
	
We recall that an operator has \emph{purely discrete} spectrum if its spectrum is a discrete set and consists of eigenvalues of finite multiplicity. Then, we have the following.
	
	\begin{theorem}
$\bar{\mathcal{L}}_t$ has purely discrete spectrum for all $t>0$.
	\end{theorem}
\begin{proof}
Consider the operator $U_t\, \mathcal{L}_{t} \, U_t^{-1}$ on $L^2(\lambda)$ obtained by conjugating $\mathcal{L}_{t}$ with the isometry $U_t\colon L^2(\lambda_{t}) \rightarrow L^2(\lambda)$ given by $U_t f=f\sqrt{h_{t}}$. Simple computations, see e.g.~\cite[Sec.\ 7]{BC3}, lead to
\begin{align*}
 U_t \mathcal{L}_{t}\, U_t^{-1}=\Ls+V_t, \qquad   V_{t}= -\frac{1}{4} \frac{|\nabla h_t|^2}{h_t^2}-\frac{1}{2}\frac{\mathcal{L}h_t}{h_t}
\end{align*}
on $C_{c}^{\infty}$, where $V_t$ has to be meant as a multiplication operator by the function $V_{t}$. Observe that $V_{t}$ is smooth since $h_{t}$ is smooth and does not vanish.

By Theorem~\ref{teo:main}~(3), as $r \to \infty$,
\[
     V_{t}(r) \sim -\frac{1}{16} \left(\frac{r}{t}\right)^2+\frac{1}{8}\left(\frac{r}{t}\right)^2=\frac{1}{16}\left(\frac{r}{t}\right)^2.
\]
In other words, $V_t(r)\asymp  \left(\frac{r}{t}\right)^2$ for $r \to \infty$ (we notice that this is the same behavior of the sub-Riemannian analogs of $V_{t}$ on H-type groups, see~\cite[Proposition 5.3]{BC2}, and of its Euclidean version on $\R^{d}$; we wonder whether this is a manifestation of a more general result).  This implies that $V_{t}$ is bounded from below, whence $\Ls+V_{t}$ is essentially self-adjoint on the domain $C_{c}^{\infty}$ by~\cite[Proposition 2.2]{BC3}. Moreover, its closure $\overline{\Ls+V_{t}}$ has purely discrete spectrum by~\cite[Proposition 4.6]{BC3}. But by~\cite[Proposition 7.2]{BC3}, $\overline{\Ls+V_{t}}$ and $\bar{\mathcal{L}}_t$ are unitarily equivalent, whence the spectrum of $\bar{\mathcal{L}}_{t}$ is discrete as well.
\end{proof}
    
    \subsection*{Acknowledgements}
We wish to thank Maria Vallarino for fruitful discussions on Damek--Ricci spaces, and Jean-Philippe Anker for pointing out the reference~\cite{GM} which allowed us to state Theorem~\eqref{teo:main}~(3) in its present form. 

This work was initiated when the first-named author was a postdoctoral fellow of the Research Foundation -- Flanders (FWO) at Ghent University, Belgium, under the postdoctoral grant 12ZW120N. The second-named author is member of the project  ``Harmonic analysis on continuous and discrete structures'' funded by Compagnia di San Paolo (CUP E13C21000270007). Both authors were also partially supported by the INdAM--GNAMPA 2022 Project ``Generalized Laplacians on continuous and discrete structures'' (CUP\_E55F22000270001).

\end{document}